\documentclass[12pt,a4paper,final]{amsart}  
\usepackage{geometry, comment}
\usepackage{color}  
\usepackage{graphicx}
\usepackage{amssymb}

\newtheorem{thm}{Theorem}[section]     
\newtheorem{lem}[thm]{Lemma}
\newtheorem{prop}[thm]{Proposition}
\newtheorem{cor}[thm]{Corollary}
\newtheorem{obs}[thm]{Observation}
\theoremstyle{definition}  
\newtheorem{definition}[thm]{Definition}

\theoremstyle{remark}
\newtheorem{remark}{{\bf Remark}}[section]
\numberwithin{equation}{section}

\DeclareMathOperator*{\osc}{osc}
\DeclareMathOperator{\USC}{USC}
\DeclareMathOperator{\LSC}{LSC}

\DeclareMathOperator{\tr}{tr}

\DeclareMathOperator{\Lip}{Lip}

\newcommand{\tim}{\times}   
\newcommand{\R}{\mathbb R}
\newcommand{\T}{\mathbb T}
\newcommand{\Z}{\mathbb Z}
\newcommand{\pl}{\partial}

\newcommand{\cA}{\mathcal{A}}
\newcommand{\cF}{\mathcal{F}}
\newcommand{\N}{\mathbb N}

\newcommand{\lan}{\langle}
\newcommand{\ran}{\rangle}

\newcommand{\cM}{\mathcal{M}}

\newcommand{\gth}{\theta}
\newcommand{\cR}{\mathcal{R}}

\newcommand{\gd}{\delta}
\newcommand{\gs}{\sigma}

\newcommand{\fr}{\frac}
\newcommand{\gO}{\varOmega}

\def\gl{\lambda}

\def\ga{\alpha}
\def\gl{\lambda}

\def\gz{\zeta}
\def\mid{\,:\,}
\newcommand{\lbar}[1]{\mkern 1.9mu\overline{\mkern-1.9mu#1\mkern-0.1mu}
\mkern 0.1mu}
\def\ol{\lbar}

\usepackage{amsrefs}

\newcommand{\bye}{\end{document}}
\newcommand{\by}{\end{proof}\end{document}}
\def\erf{\eqref}

\def\cL{\mathcal{L}}
\def\bS{\mathbb{S}}
\def\cP{\mathcal{P}}
\def\ep{\varepsilon}
\def\cU{\mathcal{U}}

\def\cG{\mathcal{G}}
\def\cF{\mathcal{F}}
\def\cR{\mathcal{R}}

\def\t{\T^n}

\def\b{\t\tim\cA}
\def\ran{\rangle}
\def\lan{\langle}

\begin{document}

\title[The vanishing discount problem and viscosity Mather measures. 
Part 1]
{The vanishing discount problem and viscosity Mather measures. 
Part 1: the problem on a torus}

\author[H. ISHII, H. MITAKE, H. V. TRAN]
{Hitoshi Ishii${}^*$, Hiroyoshi Mitake and Hung V. Tran}

\thanks{
The work of HI was partially supported by the JSPS grants: KAKENHI \#23244015, 
\#26220702, \#23340028,\#16H03948, 
the work of HM was partially supported by the JSPS grants: KAKENHI \#15K17574, \#26287024, \#23244015, \#16H03948,
and the work of HT was partially supported in part by NSF grant DMS-1615944. 
}

\address[H. Ishii]{
Faculty of education and Integrated Arts and Sciences,
Waseda University
1-6-1 Nishi-Waseda, Shinjuku, Tokyo 169-8050 Japan}
\email{hitoshi.ishii@waseda.jp}
\thanks{${}^*$ Corresponding author}

\address[H. Mitake]{
Institute of Engineering, Division of Electrical, Systems and Mathematical Engineering, 
Hiroshima University 1-4-1 Kagamiyama, Higashi-Hiroshima-shi 739-8527, Japan}
\email{hiroyoshi-mitake@hiroshima-u.ac.jp}

\address[H. V. Tran]
{
Department of Mathematics, 
480 Lincoln Dr, Madison, WI 53706, USA}
\email{hung@math.wisc.edu}

\date{\today}

\maketitle

\begin{abstract}
We develop a variational approach to the vanishing discount problem for 
fully nonlinear, degenerate elliptic, partial differential equations.  
Under mild assumptions, we introduce viscosity Mather measures for such partial differential equations, 
which are natural extensions of the Mather measures.   
Using the viscosity Mather measures, 
we prove that the whole family of solutions $v^\gl$ of the 
discount problem with the factor $\gl>0$ 
converges to a solution of the ergodic problem as $\gl\to0$.
\end{abstract} 

\tableofcontents 

\section{Introduction and Main Result}
We study 
the fully nonlinear, possibly  
degenerate, elliptic partial differential equation (PDE for short)
\begin{equation}\label{DP}\tag{DP} 
\gl u(x)+F(x,Du(x),D^2u(x))=0 \ \ \text{ in }\ \t,   
\end{equation}
where $\gl$ is a given positive constant, which we call a \textit{discount factor} 
in view of stochastic optimal control and differential games, 
$F\mid \t\tim\R^n\tim\bS^n\to\R$ is a given continuous function. 
Here, $\t$ and $\bS^n$ denote the standard $n$-dimensional torus, i.e., $\t=\R^n/\Z^n$, 
and the space of $n\tim n$ real symmetric matrices, respectively. 
The function $u$ represents a real-valued unknown function on $\t$,    
and $Du$ and $D^2u$ denote the gradient and Hessian of $u$, respectively. 
We are always concerned here with viscosity solutions of fully nonlinear, possibly degenerate,  
elliptic PDE, and the adjective ``viscosity'' is omitted 
throughout the paper.

In this paper, we study the 
\textit{vanishing discount problem} for \erf{DP}, that is, 
for solutions $v^\gl$ of \erf{DP}, with $\gl>0$, we investigate the asymptotic 
behavior of $\{v^\gl\}_{\gl>0}$ as $\gl\to 0$. 
Associated with the vanishing discount problem for \erf{DP} is the  
\textit{ergodic problem} (or the \textit{additive eigenvalue problem}): 
\begin{equation}\label{EP}\tag{EP}
F(x,Du(x),D^2u(x))=c \ \ \text{ in }\ \t,
\end{equation}
where the unknown is a pair of a function $u\in C(\t)$ and a constant $c\in\R$ 
such that $u$ is a solution of \erf{EP}. If $(u,c)\in C(\t)\tim\R$ 
is a solution of \erf{EP}, then $c$ is called a \textit{critical value} (or an 
\textit{additive eigenvalue}). 

There are many results in the literature on the vanishing discount problem  
and the following classical observation is fundamental and well-known.  
\begin{obs}\label{fund}   
Under appropriate hypotheses on $F$, for each $\gl>0$, 
\erf{DP} has a unique solution $v^\gl\in C(\t)$,  
\erf{EP} has a solution $(u,c)\in C(\t)\tim\R$, and the constant $c$ is determined uniquely. Furthermore, 
\begin{equation}\label{i-1}
c=-\lim_{\gl\to 0+}\gl v^\gl(x)\ \ \text{ in }C(\t),
\end{equation}
and, for any sequence $\{\gl_j\}_{j\in\N}\subset(0,\,\infty)$ converging to zero, there exists its subsequence, which is denoted again by the same symbol, such that $\{v^{\gl_j}-\min_{\t}v^{\gl_j}\}_{j\in\N}$ converges 
in $C(\t)$, and the pair 
of the function $u\in C(\t)$ given by 
\begin{equation}\label{i-2}
u(x):=\lim_{j\to\infty}(v^{\gl_j}(x)-\min_{\t}v^{\gl_j}),  
\end{equation}
and the constant $c$ is a solution of \erf{EP}.
\end{obs}

This was first formulated in this generality in \cite{LPV}, in the context of periodic homogenization of first-order Hamilton-Jacobi equations. 
Unstated assumptions for the observation above 
include the validity of the comparison principle for solutions (subsolutions, and supersolutions) of \erf{DP} and the equi-continuity  
of the family $\{v^\gl\}_{\gl>0}$ on $\t$, which will be clearly stated later.   

One of current interests in the vanishing discount problem is the question of the convergence of the \emph{whole family} $\{v^\gl\}$ 
of \erf{DP} to a unique limit.  That is,  
the question is whether the whole family $\{v^\gl\}_{\gl>0}$ (after a normalization)
converges, or not, to a function in $C(\t)$ as $\gl\to 0$. 
The real difficulty comes from the fact that \erf{EP} in general has many solutions even up to additive constants (see 
\cite[Remark 1.2]{BCD1997}, 
\cite[Remarks]{LPV}, \cite[Example 1]{MiTr} for example)
and a priori, the limit \erf{i-2} might depend on the subsequence $\{\lambda_i\}$.
The convergence results of the whole family $\{v^\gl\}$ were recently established for convex Hamilton-Jacobi equations in \cite{DFIZ} (first-order case), \cite{AlAlIsYo} (first-order case with Neumann-type boundary condition), \cite{MiTr} (degenerate viscous case).
In the case of first-order equations \cite{DFIZ,AlAlIsYo}, optimal control formulas are used and the \textit{Mather measures}, 
introduced and developed in \cite{Mat, Man}, play a central role. 
The proof in \cite{MiTr} for a degenerate viscous case is based on the nonlinear adjoint method, 
stochastic Mather measures and a regularization procedure by mollifications to handle the delicate viscous term using a commutation lemma.

In this paper, we first establish a new variational formula 
for the critical value $c$, 
based on a minimax theorem due to Sion \cite{Si}, 
which is inspired by the duality principle in \cite{Go}, 
and prove, in the line of \cite{DFIZ}, 
the convergence of whole family $\{v^\gl\}$ for fully nonlinear, (possibly degenerate) elliptic PDEs. 
In order to present our approach to the vanishing discount problem  
in a simple way, we consider the equations on the torus $\t$.  
In a forthcoming paper \cite{IsMiTr2}, 
we investigate the vanishing discount problem for fully nonlinear PDEs in a bounded domain under various boundary conditions.

\subsection{Assumptions}
We first specify the function $F$ here, and state the precise assumptions.  
Let $\cA$ be a non-empty, $\gs$-compact and locally compact metric space and 
$F\mid \t\tim\R^n\tim\bS^n\to\R$ be given by
\begin{equation}\tag{F1}\label{F1}
F(x,p,X)=\sup_{\ga\in\cA}(-\tr a(x,\ga)X-b(x,\ga)\cdot p-L(x,\ga)), 
\end{equation}
where  
$a\mid \t\tim\cA \to\bS_+^n$, $b\mid \t\tim\cA\to\R^n$ and $L\mid \t\tim\cA\to\R$ are continuous.
Here, $\bS_+^n$ denotes the set of all non-negative definite matrices 
$A\in\bS^n$, $\tr A$ and $p\cdot q$ designate the trace of $n\tim n$ matrix $A$ and
the Euclidean inner product of $p,q\in\R^n$, respectively. 

Next, we assume
\begin{equation}\tag{F2}\label{F2}
F\in C(\t\tim\R^n\tim\bS^n).
\end{equation}
It is clear under \erf{F1} and \erf{F2} 
that $F$ is degenerate elliptic in the sense 
that for all $(x,p,X)\in \t\tim\R^n\tim\bS^n$, if $Y\in\bS_+^n$,  
then $F(x,p,X+Y)\leq F(x,p,X)$ and that, for each $x\in\t$, 
the function $(p,X)\mapsto F(x,p,X)$ is convex on $\R^n\tim\bS^n$. 
The equations \erf{DP}, \erf{EP}, with $F$ of the form \erf{F1}, 
are called Bellman equations, or Hamilton-Jacobi-Bellman equations 
in connection with the theory of optimal control. 
In this viewpoint, the set $\cA$ is called a control set or region.    

In Section \ref{example}, we give some examples of PDEs, to which 
our main results (Theorems \ref{thm2-sec3} and \ref{thm1-sec2} below) apply. 
All PDEs taken up there have the form \erf{F1} and the control sets 
$\cA$ are closed subsets of Euclidean spaces. We, however, believe 
that it is useful to formulate our results 
without assuming the Euclidean structure of $\cA$. For instance, in the application to relaxed controls, the space of probability measures 
takes the place of the control set.   

We here describe the fundamental settings which we use throughout the paper, 
and put labels on these for convenience later. 
\begin{equation}\tag{CP}\label{CP}
\left\{\text{
\begin{minipage}{0.85\textwidth}
The comparison principle holds for \erf{DP}. 
More precisely, let $\gl>0$, $\phi\in C(\b)$, and  
$U$ be any open subset of $\t$. If  
$v\in \USC(U),\,w\in \LSC(U)$ are a subsolution and a supersolution of 
$\gl u+F(x,Du,D^2u)=0$ in $U$, respectively, and $v\leq w$ on $\pl U$, then 
$v\leq w$ in $U$.
\end{minipage}
}\right.\end{equation}
 
\begin{equation} \tag{EC}\label{EC}
\left\{\text{
\begin{minipage}{0.85\textwidth}For $\gl>0$, let $v^\gl$ be a solution of \erf{DP}. The family $\{v^\gl\}_{\gl>0}$ is equi-continuous on $\t$.
\end{minipage}
}\right.
\end{equation}
 
\begin{equation}\tag{\text{L}}\label{L} 
\left\{
\text{
\begin{minipage}{0.85\textwidth}
There exists a non-empty 
compact set $K_0\subset\cA$ 
such that 
\[
\inf_{\t\tim(\cA\setminus K_0)}L \geq L_0,
\]
where $ \displaystyle L_0=\max_{x\in\t}\inf_{\ga\in \cA}L(x,\ga)=-\min_{x\in \t} F(x,0,0).$
\end{minipage}
}\right.\end{equation}

Condition \erf{L} is always satisfied when $\cA$ is compact. Indeed, 
if $\cA$ is compact and if we set $K_0:=\cA$ and  $L_0:=
\max_{x\in\t}\min_{\ga\in K_0}L(x,\ga)$, then we have
\[
\max_{x\in\t}\min_{\ga\in K_0}L(x,\ga)\, = \, L_0\,
\leq\ \inf_{\t\tim(\cA\setminus K_0)}L=\inf_{\emptyset}L=+\infty. 
\]  
Secondly, in case $\cA$ is not compact, the {\it coercivity} of $L$ implies \erf{L}.
Let us recall that $L$ is coercive if
\[
L=+\infty \quad \text{at infinity.}    
\]
 
Sometimes, we normalize the critical value $c$ to be zero and 
we put this as an assumption to conveniently recall later. 
\begin{equation}\tag{Z}\label{Z}
\text{
\begin{minipage}{0.85\textwidth}
Problem \erf{EP} has a solution $(w,c)\in C(\t)\tim\R$ and 
$c=0$. 
\end{minipage}
}\end{equation}

\subsection{Main results} 
We first state the fundamental results. 
\begin{prop} \label{exist1}Assume \erf{F1}, \erf{F2} 
and \erf{CP}. \emph{(i)} Let $\gl>0$. 
Problem \erf{DP} has a unique solution $v^\gl$ in $C(\t)$. 
Moreover, there exists a constant $M_0>0$ such that 
$\gl |v^\gl|\leq M_0$ on $\t$ for all $\gl>0$.
\emph{(ii)} Let $c_1,\,c_2\in\R$ 
and $U$ be an open subset of $\t$. Let $v,\,w\in C(\ol U)$ be, respectively,  
a subsolution  of $F(x,Dv,D^2v)=c_1$ in $U$ and a supersolution of 
$F(x,Dw,D^2 w)=c_2$ in $U$. If $c_1<c_2$ and $v\leq w$ on $\pl U$, 
then $v\leq w$ in $U$. 
\end{prop}

\begin{prop}\label{prop-fund} Assume \erf{F1}, \erf{F2}, \erf{CP},
and \erf{EC}. Then the conclusions of Observation {\rm \ref{fund}} hold. 
\end{prop}

These results are quite classical, but we give proofs in Appendix for 
the self-completeness. 
We then address our main result. 

\begin{thm}\label{thm2-sec3} Assume \erf{F1}, \erf{F2}, \erf{CP''}, \erf{EC} and \erf{L}.
Let $c$ be the critical 
value of \erf{EP} and, for each $\gl>0$, let $v^\gl\in C(\t)$ be the unique solution of \erf{DP}. 
Then, the family $\{v^\gl+\gl^{-1}c\}_{\gl>0}$ converges to a function $u \in C(\t)$ as $\gl\to 0$. 
Furthermore, the pair $(u,c)$ is a solution of \erf{EP}. 
\end{thm}

We note that we need to assume \erf{CP''}, a comparison principle which is similar to but stronger than \erf{CP}, in Theorem \ref{thm2-sec3}. 
Assumption \erf{CP''} will be stated precisely in Section \ref{sec3}. 


\subsection*{Outline of the paper} 
We first consider the case where  the control set $\cA$ is compact
to demonstrate the key ideas of our approach in Section \ref{sec2} 
and prove the main result of the section, Theorem \ref{thm1-sec2}, which is a simpler version of Theorem \ref{thm2-sec3}.
We then study the general case and give the proof of Theorem \ref{thm2-sec3} in Section \ref{sec3}.
Finally, some examples, which satisfy all assumptions of Theorem \ref{thm2-sec3} or Theorem \ref{thm1-sec2}, are discussed in Section \ref{example}. 
In Appendix, the proofs of Propositions \ref{exist1} and \ref{prop-fund} are given, 
which are quite standard in the theory of viscosity solutions.

\subsection*{Notation} 
We often write $F[u]$ to denote the function 
$x\mapsto F(x,Du(x),D^2u(x))$. 
Given a metric space $E$, $\Lip(E)$ denotes the space of Lipschitz continuous functions on $E$. Also, let $C(E)$ (resp., $\LSC(E)$) denote the set of all upper semicontinuous (resp., lower semicontinuous) functions on $E$, and let $C_{\rm c}(E)$ denote the space of continuous 
functions on $E$ with compact support.

\section{Compact control set} \label{sec2}
In this section, in order to present the main ideas of the variational approach in a simple way  
we assume further that
\[\tag{F3}\label{F3}
\text{
\begin{minipage}{0.85\textwidth}
the control set $\cA$ is compact. 
\end{minipage}
} 
\] 
We will remove this assumption in the next section. 

Before stating the main theorem of this section we give some notions. 
For $\phi\in C(\b)$, we define $F_\phi\in C(\t\tim\R^n\tim\bS^n)$ by  
\[
F_\phi(x,p,X)=\max_{\ga\in\cA}(-\tr a(x,\ga)X-b(x,\ga)\cdot p-\phi(x,\ga)),
\]
and consider the problems
\begin{equation}\tag{DP$_\phi$}\label{DP'}
\gl u(x)+F_\phi(x,Du(x),D^2u(x))=0 \ \ \text{ in }\ \t,
\end{equation}
and
\begin{equation}\tag{EP$_\phi$}\label{EP'}
F_\phi(x,Du(x),D^2u(x))=0 \ \ \text{ in }\ \t.
\end{equation}
Note that, since $\b$ is compact, the function $F_\phi$ is 
continuous in $\t\tim\R^n\tim\bS^n$.  
We write $F_\phi[u]$ to denote the function 
$x\mapsto F_\phi(x,Du(x),D^2u(x))$. 

Here is the main theorem of this section. 
\begin{thm} \label{thm1-sec2} 
Assume \erf{F1}, \erf{F2}, {\rm(F3)}, and \erf{EC}. 
Moreover, we assume 
\begin{equation}\tag{CP$'$}\label{CP'}
\left\{\text{
\begin{minipage}{0.85\textwidth}
the comparison principle holds for \erf{DP'}. 
More precisely, let $\gl>0$, $\phi\in C(\b)$, and  
$U$ be any open subset of $\t$. If  
$v\in \USC(U),\,w\in \LSC(U)$ are a subsolution and a supersolution of 
$\gl u+F_\phi[u]=0$ in $U$, respectively, and $v\leq w$ on $\pl U$, then 
$v\leq w$ in $U$.
\end{minipage}
}\right.\end{equation}
Let $c$ be the critical 
value of \erf{EP} and, for each $\gl>0$, let $v^\gl\in C(\t)$ be a (unique) solution of \erf{DP}. 
Then, the family $\{v^\gl+\gl^{-1}c\}_{\gl>0}$ is convergent in $C(\t)$ as $\gl\to 0$. 
Furthermore, if we set 
\[
u(x)=\lim_{\gl\to 0+}(v^\gl(x)+\gl^{-1}c) \ \ \text{ for }x\in\t,
\] 
then $(u,c)$ is a solution of \erf{EP}. 
\end{thm}
Note that assumption \erf{CP'} implies 
\erf{CP}, and thus  
 by Proposition \ref{prop-fund} that \erf{EP} has a solution $(u,c)\in C(\t)\tim\R$ 
and the critical value $c$ is unique. 
Let $v^\gl$ be the solution of \erf{DP} for each $\gl>0$. 
If we set $\widetilde F=F-c$ on $\t\tim\R^n\tim\bS^n$, then the function 
$v:=v^\gl+\gl^{-1}c$ is a solution of $\gl v+ \widetilde F[v]=0$ in $\t$. 
Moreover, if $u$ is a solution of \erf{EP}, then $u$ is a solution of $\widetilde F[u]=0$ in $\t$. These observations allow us to 
assume, without loss of generality, that the critical value  $c$ 
is zero in Theorem \ref{thm1-sec2}. 
In short, we can assume (Z) without loss of generality.    

\subsection{Viscosity Mather measures}

Inspired by the generalization of Mather measures to second-order elliptic PDEs in \cite{Go},
we introduce viscosity Mather measures for \erf{EP} and \erf{DP}.  
In the following discussion, we implicitly assume that \erf{Z} holds.

We define the sets $\cF_0\subset C(\b)\tim C(\t)$, $\cG_0\subset C(\b)$, respectively, 
by 
\begin{align*}
&\cF_0:=\{(\phi,u)\in C(\b)\tim C(\t)\mid u \ \text{is a subsolution of} \ \erf{EP'}\}, \\
&\cG_0:=\{\phi\in C(\b)\mid (\phi,u)\in \cF_0\,\text{ for some }\,u\in C(\t)\}.
\end{align*}
Note that the definition of
$\cF_0$ depends on $a$ and $b$, but not on $L$.  
Moreover, $\cF_0\not=\emptyset,  \cG_0\not=\emptyset$ in view of \erf{Z}. 

Next, let $\cR$ denote the space of Radon measures on $\T^n\tim\cA$  
and $\cP$ denote the space of Radon probability measures on $\T^n\tim\cA$. 
The Riesz representation theorem ensures that the dual space of $C(\T^n\tim\cA)$
is identified with $\cR$. In this viewpoint, we write 
\[
\lan \mu,\phi\ran=\int_{\b}\phi(x,\ga)\,\mu(dxd\ga) \ \ \text{ for }\phi\in C(\b) \text{ and } 
\mu\in\cR.
\]

Note that $\b$ can be regarded as a separable metric space, 
and by the Prokhorov theorem 
that, due to the compactness of $\cA$, 
$\cP$ is sequentially compact, that is,
given a sequence $\{\mu_j\}\subset\cP$, we can choose a subsequence of $\{\mu_j\}$ 
which converges to some $\mu\in\cP$ weakly in the sense of measure.  

\begin{lem} \label{thma2} Under hypotheses \erf{F1}, \erf{F2} and 
\erf{CP'}, 
the set $\cG_0$ is a convex cone in $C(\T^n\tim\cA)$,
with vertex at the origin. 
\end{lem}

Formally, it is easy to see that $\cG_0$ is convex. Indeed, for any 
$(\phi^0,u^0), (\phi^1,u^1)\in\cF_0$ such that $u^0, u^1 \in C^2(\t)$, 
setting $\phi^t:=t\phi^1+(1-t)\phi^0$, $u^t:=t u^1+(1-t)u^0$ for $t\in[0,1]$, 
we have in the classical sense that
\[
F_{\phi^t}[u^t]\le (1-t)F_{\phi^0}[u^0]+tF_{\phi^1}[u^1]\le 0. 
\]
We give a rigorous proof in the next subsection (see the proof of Lemma \ref{dmmc-thm1}).

Let ${\cG_0}'\subset\cR$ denote the dual cone of $\cG_0$, that is, we set 
\[
{\cG_0}':=\left\{\mu\in\cR\mid
\lan \mu,\phi\ran\geq 0 \ \ \text{ for all }
\phi\in\cG_0\right\}.
\] 

\begin{prop} Let $\mu\in{\cG_0}'$. Then,  
\begin{equation}\label{closing}
\int_{\b}\cL_\ga\psi(x)\mu(dxd\ga)=0 
\quad\text{for all} \ \psi\in C^2(\t),  
\end{equation}
where $\cL_\ga\mid C^2(\T^n)\to C(\T^n)$, with $\ga\in\cA$, denotes the linear operator given by 
\[
\cL_\ga\psi(x)=-\tr a(x,\ga)D^2\psi(x)-b(x,\ga)\cdot D\psi(x).
\]
\end{prop} 
\begin{proof} Let $\psi\in C^2(\t)$. 
If we set $\phi(x,\ga)=\cL_\ga\psi(x)$ for $(x,\ga)\in\b$, then 
we have $F_\phi[\psi]=0$ in $\t$ in the classical sense 
and, hence, 
$(\phi,\psi)\in\cF_0$. Similarly, we have $(-\phi,-\psi)\in\cF_0$. 
These imply that $\pm \phi \in \cG_0$.
Thus, we have $\lan \mu,\pm\phi\ran\geq 0$ and, therefore, $\lan \mu,\phi \ran=0$.     
\end{proof}

The main claim in this subsection is the following. 
\begin{thm} \label{thm2-sec2} Assume \erf{F1}, \erf{F2}, \erf{CP'} and \erf{Z}. Then, 
\begin{equation}\label{rep1-sec2}
\min_{\mu\in\cP\,\cap\,{\cG_0}'}\int_{\b} L(x,\ga)\mu(dxd\ga)=0. 
\end{equation}
\end{thm}

We give a proof of Theorem \ref{thm2-sec2} in 
the next subsection (see the proof of Theorem \ref{thm3-sec2}). 
Note that the minimizing problem \eqref{rep1-sec2} 
has a minimizer over the convex cone $\cP\,\cap\,{\cG_0}'$,  
since $L\in C(\b)$ and  $\cP\,\cap\,{\cG_0}'$ is compact in the topology of weak convergence in measure. 

In the general case where \erf{Z} is not assumed, 
the proposition below, which is a direct consequence of Theorem \ref{thm2-sec2}, gives a formula for the critical value $c$.  

\begin{cor} \label{cor1} Assume \erf{F1}, \erf{F2}, \erf{CP'}, and 
that \erf{EP} has a solution $(u,c)\in C(\t)\tim\R$. Then, 
\begin{equation}\label{cor1-1}
c=-\min_{\mu\in\cP\,\cap\,{\cG_0}'}\int_{\b} L(x,\ga)\mu(dxd\ga). 
\end{equation}
\end{cor}

The variational formula above   
for the additive eigenvalue $c$ (see also Corollary \ref{cor2} below) 
can be considered as an additive version of a variational 
formula \cite{DV1, DV2, Ar} for the principal ``multiplicative''  eigenvalue for second-order elliptic operators.

\begin{definition}[Viscosity Mather measure]
We define $\cM_{L,0}\subset\cP$ as the set of all measures $\mu\in \cP\,\cap\,{\cG_0}'$ which solves minimizing problem \eqref{rep1-sec2}. We call any measure $\mu\in\cM_{L,0}$ 
a \emph{viscosity Mather measure}.  
\end{definition}

\begin{remark}\ 
\begin{itemize}
\item[{\rm(i)}] 
The variational problem \eqref{rep1-sec2} can be considered as a natural extension of Mather's minimizing problem 
(see \cite{Mat}) for second-order PDEs. 
This kind of extension was first addressed in \cite{Go}. 
\item[{\rm(ii)}] 
The identity \eqref{closing} corresponds to the closedness 
property for the original Mather measures. A measure satisfying \eqref{closing} is sometimes 
called a \textit{closing measure} or a \textit{holonomic measure}. 
This is a ``weak" flow invariance condition for the associated Euler-Lagrange flow for the first-order Hamilton-Jacobi equations 
(see \cite{Man}). 
\end{itemize}
\end{remark}

\subsection{Discounted Mather measures}  \label{dmmc}
It is convenient in what follows that the discount factor $\gl$ in \erf{DP} or \erf{DP'} can be taken as zero.
For $\gl\geq 0$, we 
define $\cF_\gl\subset C(\T^n\tim\cA)\tim C(\t)$ by 
\[
\cF_\gl:=\{(\phi,u)\in C(\b)\tim C(\t)\mid u \ \text{is a subsolution of} \ \erf{DP'}\}. 
\]
\begin{lem}\label{dmmc-thm2} Let $\gl\in[0,\,\infty)$. 
Under hypotheses \erf{F1}, \erf{F2} and 
\erf{CP'}, the set $\cF_\gl$ is convex.  
\end{lem}  
  
\begin{proof} 
Let $(\phi^0,u^0),\,(\phi^1,u^1)\in\cF_\gl$ and $t\in(0,\,1)$, and  
set $\phi^t:=t\phi^1+(1-t)\phi^0$ and $u^t:=tu^1+(1-t)u^0$. 
We prove that $\gl u^t+F_{\phi^t}[u^t]\leq 0$ in $\t$ in the viscosity sense.

Let $\psi\in C^2(\t)$ and $z\in\t$ be such that $u^t-\psi$ has a strict maximum at 
$z$. 
We may assume that $\max_{\t}(u^t-\psi)=0$. Then, $u^t(z)=\psi(z)$. 
Suppose that $\gl u^t(z)+F_{\phi^t}[\psi](z)>0$, and show a contradiction. 

We choose $\gd>0$ and an open neighborhood $U$ of $z$ 
so that 
\begin{equation}\label{super-psi}
\gl \psi(x)+F_{\phi^t}[\psi](x)>\gd \ \ \text{ for all }x\in U.  
\end{equation}
Set 
\[
w(x)=-t^{-1}(1-t)u^1(x)+t^{-1}\psi(x) \ \ \text{ for }x\in\t,
\]
and intend to show that $w$ is a supersolution of 
\begin{equation}\label{superU}
\gl w+F_{\phi^0}[w]= t^{-1}\gd \ \ \ \text{ in }U. 
\end{equation}
Once this is done, we easily get a contradiction. Indeed, 
we apply the comparison principle \erf{CP'} to $u^0$ and $w$, 
to find that 
\[
\sup_{U}(u^0-w)\leq \max_{\pl U}(u^0-w), 
\]
but this is a contradiction since $t(u^0-w)=u^t-\psi$ attains a strict maximum at $z\in U$. 

To prove the viscosity property \erf{superU} of $w$, we fix $\eta\in C^2(U)$ 
and $y\in U$, and assume that $w-\eta$ takes a minimum at $y$. 
This implies that 
the function    
\[ 
-t(1-t)^{-1}(w-\eta)=u^1-(1-t)^{-1}\psi+t(1-t)^{-1}\eta
\]
takes a maximum at $y$. By the viscosity property of $u^1$, we get 
\begin{equation}\label{psi-eta}
\gl u^1(y)+F_{\phi^1}[(1-t)^{-1}\psi-t(1-t)^{-1}\eta](y)\leq 0. 
\end{equation}

We compute that for any $\ga\in\cA$,
\[\begin{aligned}
\cL_\ga \psi(y)&-t\phi^0(y,\ga)-(1-t)\phi^1(y,\ga)
\\&\,=t\left(\cL_\ga\eta(y)-\phi^0(y,\ga)\right)
+\cL_\ga\psi(y)-t\cL_\ga\eta(y)-(1-t)\phi^1(y,\ga)
\\&\,=t\left(\cL_\ga\eta(y)-\phi^0(y,\ga)\right)
\\&\,\quad+(1-t)\left\{(1-t)^{-1}\cL_\ga\psi(y)-t(1-t)^{-1}\cL_\ga\eta(y)
-\phi^1(y,\ga)\right\},
\end{aligned}\]
from which we deduce that 
\[
F_{\phi^t}[\psi](y)
\leq tF_{\phi^0}[\eta](y)+(1-t)
F_{\phi^1}[(1-t)^{-1}\psi-t(1-t)^{-1}\eta](y).
\]
Thus, using \erf{super-psi}, we get
\[\begin{aligned}
\gd&\leq \gl \psi(y)+tF_{\phi^0}[\eta](y)
+(1-t)F_{\phi^1}[(1-t)^{-1}\psi-t(1-t)^{-1}\eta](y)
\\&=\gl (tw(y)+(1-t)u^1(y))+ tF_{\phi^0}[\eta](y)
+(1-t)F_{\phi^1}[(1-t)^{-1}\psi-t(1-t)^{-1}\eta](y)
\\&=t\left\{\gl w(y)+F_{\phi^0}[\eta](y)\right\}
+(1-t)\left\{\gl u^1(y)+F_{\phi^1}[(1-t)^{-1}\psi-t(1-t)^{-1}\eta](y)\right\}.
\end{aligned}
\]
Combining this with \erf{psi-eta} yields
\[
t^{-1}\gd\leq \gl w(y)+F_{\phi^0}[\eta](y). 
\] 
This shows that $w$ is a supersolution of \erf{superU}, which completes the proof. 
\end{proof}

For $(z,\gl)\in\t\tim[0,\,\infty)$, 
we define $\cG_{z,\gl}\subset C(\b)$ by 
\[
\cG_{z,\gl}:=\{\phi-\gl u(z)\mid (\phi,u)\in\cF_\gl \ \text{for some} \ u\in C(\t)\}.
\]  
Note that $\cG_0=\cG_{z,0}$ for all $z\in\t$.
\begin{lem} \label{dmmc-thm1} Let $\gl\in[0,\,\infty)$. 
Under hypotheses \erf{F1}, \erf{F2} and 
\erf{CP'}, 
the set $\cG_{z,\gl}$ is a convex cone in $C(\T^n\tim\cA)$,
with vertex at the origin. 
\end{lem}
We note that the lemma above includes Lemma \ref{thma2} as the special case $\gl=0$. 

\begin{proof} 
The convexity of $\cG_{z,\gl}$ is an immediate consequence 
of Lemma \ref{dmmc-thm2}.  
To see the cone property of $\cG_{z,\gl}$, 
it is enough to show that $\cF_\gl$ is a cone with vertex at the origin. 
Fix any $(\phi,u)\in\cF_\gl$ and $t\geq 0$, and observe that 
$w:=tu$ satisfies, at least in a formal level,  
\[
\gl w+\cL_\ga w =t(\gl u+\cL_\ga u)
\le
t(\gl u+F_\phi[u]+\phi)
\leq t\phi(x,\ga) 
\]
for all $(x,\ga)\in\b$, 
which can be stated as $\gl (tu)+F_{t\phi}[tu]\leq 0$ in $\t$. 
Thus, $(t\phi,tu)\in\cF_\gl$. 
This computation is easily justified in the viscosity sense 
and ensures that $\cF_\gl$ is a cone with vertex at the origin.
\end{proof} 

Our choice of the term, the viscosity Mather measures, comes from the 
fact that the stability property of viscosity solutions 
is crucial in establishing Lemma \ref{dmmc-thm1}.

Let ${\cG_{z,\gl}}{'}$ denote the dual cone of $\cG_{z,\gl}$, that is, 
\[
\cG_{z,\gl}{'}:=\{\mu\in\cR\mid \lan \mu,f\ran\geq 0
\ \ \text{ for all }f\in\cG_{z,\gl}\}. 
\]

\begin{thm} \label{thm3-sec2}
Assume \erf{F1}, \erf{F2} and \erf{CP'}. Assume, in addition, \erf{Z} 
if $\gl=0$. 
Let $\gl\geq 0$ and $v^\gl\in C(\T^n)$ be a solution of 
\erf{DP}. Then 
\begin{equation}\label{rep2-sec2}
\gl v^\gl(z)=\min_{\mu\in \cP\,\cap\,{\cG_{z,\gl}}'}\int_{\T^n\tim\cA}L(x,\alpha)\,\mu(dxd\ga).
\end{equation}
\end{thm}

Notice that formula \erf{rep1-sec2} 
can be seen as the special case of \erf{rep2-sec2}, with $\gl=0$.
Thus, the proof of Theorems \ref{thm2-sec2} is included in that of 
Theorem \ref{thm3-sec2} below. 
Note also that $v^\gl$ is unique for $\gl>0$, and is not unique in general for $\gl=0$.

\begin{proof}
We fix $(z,\gl)\in\t\tim[0,\,\infty)$.   
We note first that 
$(L,v^\gl)\in\cF_\gl$ and observe by the definition of 
$\cG_{z,\gl}{'}$ that for any $\mu\in\cP\,\cap\,\cG_{z,\gl}{'}$, 
\[
0\leq \lan \mu,L-\gl v^\gl(z)\ran=\lan \mu,L\ran-\gl v^\gl(z),
\]
which yields 
\begin{equation}\label{thm3-2-1}
\gl v^\gl(z)\leq \inf_{\mu\in\cP\,\cap\,\cG_{z,\gl}{'}}\lan\mu,L\ran. 
\end{equation}

To show the reverse inequality, 
we argue by contradiction. The first step is to suppose that
\[
\gl v^\gl(z)<\inf_{\mu\in\cP\,\cap\,\cG_{z,\gl}{'}}\lan \mu,L\ran,
\]
and select $\ep>0$ so that 
\begin{equation}\label{thm3-2-2}
\gl v^\gl(z)+\ep<\inf_{\mu\in\cP\,\cap\,\cG_{z,\gl}{'}}\lan \mu,L\ran.
\end{equation}

Now, since $\cG_{z,\gl}$ is a convex cone with vertex at the origin, we deduce that  
\[
\inf_{f\in\cG_{z,\gl}}\lan\mu,f\ran=
\begin{cases}
0 & \text{ if }\ \mu\in \cP\,\cap\,\cG_{z,\gl}{'}\\[3pt]
-\infty \ \ &\text{ if }\ \mu\in\cP\setminus \cG_{z,\gl}{'}.  
\end{cases}
\]
Accordingly, we get 
\[\begin{aligned}
\inf_{\mu\in\cP\,\cap\,\cG_{z,\gl}{'}}&\lan \mu, L\ran
=\inf_{\mu\in\cP}\left(\lan \mu, L\ran-\inf_{f\in\cG_{z,\gl}}
\lan \mu,f\ran\right)
\\&=\inf_{\mu\in\cP}\sup_{f\in\cG_{z,\gl}}\lan\mu,L-f\ran.
\end{aligned}
\]

Observe that $\cP$ is a compact convex subset of $\cR$, with topology 
of weak convergence of measures, $\cG_{z,\gl}$ is a convex subset of $C(\b)$, 
the functional $\mu\mapsto \lan\mu,L-f\ran$ is continuous and linear on $\cR$, again with topology 
of weak convergence of measures, for any $f\in C(\b)$, and the functional $f\mapsto \lan\mu,L-f\ran$ is continuous and affine on $C(\b)$ for any $\mu\in\cR$, and 
then invoke Sion's minimax theorem (see \cite[Corollary 3.3]{Si} and also \cite{Te}), to find that
\[
\inf_{\mu\in\cP}\sup_{f\in \cG_{z,\gl}}\lan \mu,L-f\ran
=\sup_{f\in \cG_{z,\gl}}\inf_{\mu\in\cP}\lan \mu,L-f\ran.
\]
Thus, we have 
\[
\inf_{\mu\in\cP\,\cap\,\cG_{z,\gl}{'}}\lan\mu,L\ran 
=\sup_{f\in\cG_{z,\gl}}\inf_{\mu\in\cP}\lan \mu, L-f\ran.
\]

The above identity and \erf{thm3-2-2} ensures that there is 
$(\phi,u)\in\cF_\gl$ such that 
\[
\gl v^\gl(z)+\ep<\inf_{\mu\in\cP}\lan\mu, L-\phi+\gl u(z)\ran. 
\]
Since the Dirac measure $\gd_{(x,\ga)}$ is in $\cP$ for all $(x,\ga)\in\b$, from the above, we get 
\[
\gl v^\gl(z)+\ep<L(x,\ga)-\phi(x,\ga)+\gl u(z) \ \  \text{ for all }
(x,\ga)\in\b,
\] 
which reads 
\[
F\le F_\phi+\gl(u-v^\gl)(z)-\ep \ \ \text{ on }\t\tim\R^n\tim\bS^n.
\]
This implies that $u$ satisfies
\begin{equation}\label{ineq-1}
\gl u+F[u]
\leq\gl u+F_\phi[u]+\gl(u-v^\gl)(z)-\ep
\le \gl(u-v^\gl)(z)-\ep \ \ \text{ in }\t. 
\end{equation} 

If $\gl>0$, then we observe that $u-(u-v^\gl)(z)+\gl^{-1}\ep$ 
is a subsolution of \erf{DP} and use \erf{CP'}, to obtain 
$u-(u-v^\gl)(z)+\gl^{-1}\ep\leq v^\gl$ on $\t$, which, evaluated at $z$, 
gives a contradiction, $\gl^{-1}\ep\leq 0$. 
If $\gl=0$, then, since $u+C$ is a subsolution of \eqref{ineq-1} 
for any $C\in\R$, by Proposition \ref{exist1} (ii), we get 
$u+C\leq v^\gl$ on $\t$ for all $C\in\R$, which is a contradiction.  
\end{proof}

\begin{definition}[Discounted Mather measure]
For $(z,\gl)\in\t\tim[0,\,\infty)$, let $\cM_{L,z,\gl}$ be the set of minimizers $\mu\in \cP\,\cap\,{\cG_{z,\gl}}'$ of the minimizing problem
\[
\min_{\mu\in\cP\,\cap\,{\cG_{z,\gl}}'}\int_{\T^n\tim\cA}L(x,\ga)\,\mu(dxd\ga).
\]  
We call any $\mu\in \cM_{L,z,\gl}$ with 
$z\in\T^n$ and $\gl>0$ a \emph{discounted Mather measure}.  
\end{definition}
Since $\cP$ is sequentially compact in the weak topology 
of measures, it is easily seen that, under the hypotheses 
of Theorem \ref{thm3-sec2}, $\cM_{L,z,\gl}\not=\emptyset$. 
Note that $\cM_{L,z,0}$ is identical to $\cM_{L,0}$ and does not depend on $z$.

\subsection{Convergence with vanishing discount} We begin with two lemmas.

\begin{lem}\label{rcom}
Assume \erf{F1}, \erf{F2}, \erf{CP'} and \erf{Z}. 
For each $\gl>0$, let 
$v^\gl\in C(\t)$ be the unique solution of \erf{DP}. Then 
$\{v^\gl\}_{\gl>0}$ is uniformly bounded on $\t$.
\end{lem}

\begin{proof} Let $u\in C(\t)$ be a solution of \erf{EP}, with $c=0$. 
Observe that the functions $u+\|u\|_{C(\t)}$ and $u-\|u\|_{C(\t)}$ are, 
respectively, 
a supersolution and a subsolution of \erf{DP}, with any $\gl>0$, 
and, by the comparison principle, \erf{CP'}, that 
$|v^\gl|\leq 2\|u\|_{C(\t)}$ on $\t$ for all $\gl>0$. This completes the proof.  
\end{proof}

\begin{lem}\label{thm7-sec2} Let $z\in\t$ and $\{\mu_j\}_{j\in\N}\subset\cR$ be 
such that, as $j\to\infty$, 
$\mu_j \to\mu$ weakly in the sense of measure for some $\mu\in\cR$
and such that  $\mu_j\in{\cG_{z,\gl_j}}'$ for all $j\in\N$ and for some 
$\{\gl_j\}\subset (0,\,\infty)$ converging to zero. 
Then $\mu\in{\cG_0}'$. 
\end{lem}

\begin{proof} Let $(\phi,u)\in\cF_0$.  
Fix any $j\in\N$, set 
$\phi_j:=\phi+\gl_j u\in C(\b)$, and observe that 
$\gl_j u+F_{\phi_j}[u]\leq 0$ in $\t$ in the viscosity sense. 
That is, we have $(\phi_j,u)\in\cF_{\gl_j}$.  
Since $\mu_j\in{\cG_{z,\gl_j}}'$, we get 
\[
0\leq \lan \mu_j,\phi_j-\gl_j u(z)\ran
=\lan \mu_j, \phi\ran+\gl_j\lan \mu_j, u\ran-\gl_j u(z).
\]
Sending $j\to\infty$ in the above yields $\lan \mu,\phi\ran\geq 0$. 
Thus, we see that $\mu\in{\cG_0}'$.  
\end{proof}

\begin{proof}[Proof of Theorem {\rm\ref{thm1-sec2}}] By Proposition \ref{prop-fund}, problem \erf{EP} has a solution $(u,c)\in C(\t)\tim\R$. 
Also, for each $\gl>0$, there exists a (unique) solution 
$v^\gl\in C(\t)$ of \erf{DP}. 

We treat first the case where $c=0$.  
Let $\cU$ be the set of accumulation points,  as 
$\gl\to 0$, in $C(\T^n)$ 
of $\{v^\gl\}_{\gl>0}$. By Lemma \ref{rcom}, 
\erf{EC} and the Arzela-Ascoli theorem, we see that 
the family $\{v^\gl\}_{\gl>0}$ is relatively compact in $C(\t)$ and, hence, 
$\cU$ is non-empty.

Next, we show that $\cU$ is a singleton by showing that
\begin{equation}\label{thm1-2-1}
v\geq w \ \ \text{ on }\t \ \ \text{ for all }
v,\,w\in\cU.
\end{equation}
To prove \erf{thm1-2-1}, we fix any 
$v,\,w\in\cU$ and $z\in\T^n$, and    
choose sequences $\{\gl_j\}_{j\in\N},\,\{\gd_j\}_{j\in\N}
\subset(0,\,1)$, 
converging to zero as $j\to\infty$, so that 
$v^{\gl_j}\to v$ and $v^{\gd_j}\to w$ in $C(\T^n)$ as $j\to\infty$. 
By Theorem \ref{thm3-sec2}, there exists
a sequence $\{\mu_j\}_{j\in\N}\subset\cP$ 
such that $\mu_j\in\cM_{L,z,\gl_j}$
for all $j\in\N$.

We may assume by passing to a subsequence 
of $\{\mu_j\}$ 
that for some $\mu$, as $j\to\infty$,
$\mu_j\to \mu$ weakly in the sense of measure. 
By Lemma \ref{thm7-sec2}, we have $\mu\in \cG_0{'}$, 
and, since $\,
\gl_j v^{\gl_j}(z)=\lan\mu_j,L\ran$, 
we get, in the limit as $j\to\infty$, $\,
\lan \mu,L\ran=0$.  
Hence, we find that $\mu\in\cM_{L,0}$. 

Noting that 
$(L-\gd_j v^{\gd_j},\,v^{\gd_j})\in\cF_0$ and 
$(L+\gl_j w,w)\in\cF_{\gl_j}$
and that $\mu\in \cM_{L,0}\subset \cG_0{'}$ 
and $\mu_j\in\cM_{L,z,gl_j}\subset\cG_{z,\gl_j}{'}$, 
we obtain  
\[
0\leq \lan\mu,L-\gd_j v^{\gd_j}\ran=-\gd_j \lan \mu,v^{\gd_j}\ran,
\]
and
\[
0\leq \lan \mu_j,L+\gl_j w-\gl_jw(z)\ran 
=\gl_j \left(v^{\gl_j}(z) +\lan\mu_j,w\ran-w(z)\right).
\]
Dividing the above inequalities by $\gd_j$ and $\gl_j$, respectively, 
and sending 
$j\to\infty$, we get 
\[
\lan\mu,w\ran\leq 0 \ \ \text{ and } \ \ w(z)\leq v(z)+\lan \mu, w\ran.
\]
These together yield $\,w(z)\leq v(z)$. Since $z\in\t$ is arbitrary, we conclude that \erf{thm1-2-1} holds and, hence, that $\cU$ is a singleton. 
Now, it is easy to see that, as $\gl \to 0+$, $v^\gl$ converges 
to a unique element of $\cU$.   

In the general case when $c$ may not be zero, if $u$ is a 
solution of \erf{EP}, then $u$ is a solution of 
$F_{L+c}[u]=0$ in $\t$. Also, the function 
$v^\gl+\gl^{-1}c$ is a solution 
of $\gl u+F_{L+c}[u]=0$ in $\t$. In view of this simple change 
of the unknown functions, we find that \erf{F1}, \erf{F2}, 
\erf{CP'} and  \erf{EC} hold, with $F$ replaced by $F_{L+c}$. 
By applying the previous observation for the case of $c=0$, 
we conclude that 
\[
\lim_{\gl\to 0+}\left(v^\gl+\gl^{-1}c\right)=u 
\ \ \text{ in }\ C(\t)
\]
for some $u\in C(\t)$ and $u$ is a solution of \erf{EP}. 
\end{proof}


\section{Proof of Theorem {\rm\ref{thm2-sec3}}} \label{sec3}
In this section we will give the proof of Theorem \ref{thm2-sec3}. 
Before going into the details, we note that 
if our equation has the form of \eqref{F1} and 
$\cA$ is compact, then the function $F(x,p,X)$ has at most 
linear growth in $|p|$ and $|X|$. 
Therefore, the function $F(x,p,X)$, defining our PDE,  
which has a higher order growth 
in $|p|$ is not handled yet in Section \ref{sec2}. 

More precisely, let us consider, for example, the equation 
\[
\gl u(x)-\tr \gs^t(x)\gs(x)D^2 u(x)+\fr 1m|Du(x)|^m=f(x) \ \ \ \text{ in }\T^n,
\]
where $m>2$, $\gl$ is a non-negative constant, 
$\gs\in \Lip(\T^n, \R^{k \times n})$ for some $k \in \N$, 
and $f\in C(\T^n)$ is a given function.  
Since
\[
\fr 1m|p|^m=\max_{q\in\R^n}\left(-q\cdot p-\fr{m-1}{m}|q|^{\fr m{m-1}}\right),
\]  
if we set   
\[
F(x,p,X):=\sup_{q\in\R^n}\left(
-\tr \gs^t(x)\gs(x) X-q\cdot p-f(x)-\fr{m-1}{m}|q|^{\fr{m}{m-1}}\right),
\]
then the equation above has the same form as \eqref{DP}. Note that 
the control set $\cA$ here is $\R^n$, which is not compact,
and the functions $b(x,q)=q$ and $L(x,q)=f(x)+\fr{m-1}{m}|q|^{\fr{m}{m-1}}$ 
play the roles of those $b$ and $L$ in the previous sections, respectively. 
We will address examples in more details in Section \ref{example}. 

In this section, we consider the general case where $\cA$ may not be 
compact. We first notice that the function 
\[
F_\phi(x,p,X)=\sup_{\ga\in\cA}\left(-\tr a(x,\ga)X-b(x,\ga)\cdot p-\phi(x,\ga)\right)
\]
may neither be continuous nor well-defined as a real-valued function in $\t\tim\R^n\tim\bS^n$ for some $\phi\in C(\b)$. 
Therefore, we need to be careful in using the strategy in Section \ref{sec2}.

Fortunately, reviewing the arguments in the previous section carefully, 
we realize that the cone $\cF_\gl$ can be replaced by a smaller convex cone.  
With this idea in mind, we introduce $\Phi^+$ as the convex cone  
consisting of functions $\phi\in C(\b)$ of the form
\[
\phi(x,\ga)=tL(x,\ga)+\chi(x), \ \ \text{ with }\ t>0 \ \text{ and }\ \chi\in C(\t),
\]
and we observe that, for any $\phi\in\Phi^+$, if 
$\phi=tL+\chi$, then 
\[\begin{aligned}
F_\phi(x,p,X)&\,=\sup_{\ga\in\cA}\left(-\tr a(x,\ga)X
-b(x,\ga)\cdot p -tL(x,\ga)-\chi(x)
\right)
\\&\,=t\sup_{\ga\in\cA}\left(-\tr a(x,\ga)t^{-1}X
-b(x,\ga)\cdot t^{-1}p -L(x,\ga)
\right)-\chi(x)
\\&\,
=tF(x,t^{-1}p,t^{-1}X)-\chi(x).
\end{aligned}
\]
This shows that if $F\in C(\t\tim\R^n\tim\bS^n)$, then 
\[
F_\phi\in C(\t\tim\R^n\tim\bS^n) \ \ \text{ for all }\phi\in\Phi^+.
\]

In this section, the following hypothesis replaces the role of \erf{CP'} in the previous section. 
\begin{equation}\tag{CP$''$}\label{CP''}
\left\{\text{
\begin{minipage}{0.85\textwidth}
The comparison principle holds for $\gl u+F[u]=\eta$, with $\eta\in C(\t)$.    
More precisely, let $\gl>0$, $\eta\in C(\t)$, and   
$U$ be any open subset of $\t$. If  
$v\in \USC(U),\,w\in \LSC(U)$ are a subsolution and a supersolution of 
$\gl u+F[u]=\eta$ in $U$, respectively, and $v\leq w$ on $\pl U$, then 
$v\leq w$ in $U$.
\end{minipage}
}\right.\end{equation}
This is obviously a stronger assumption than (CP).

\begin{lem}\label{thm1-sec3} 
Under \erf{CP''}, the following comparison principle holds: 
Let $\gl>0$, $\eta\in C(\t)$, $\phi\in\Phi^+$, and   
$U$ be any open subset of $\t$. If  
$v\in C(U),\,w\in C(U)$ are a subsolution and a supersolution of 
$\gl u+F_\phi[u]=\eta$ in $U$, respectively, and $v\leq w$ on $\pl U$, then 
$v\leq w$ in $U$.
\end{lem}

\begin{proof} Let $v\in C(U),\,w\in C(U)$ be a subsolution and a supersolution of $\gl u+F_\phi[u]=\eta$ in $U$, respectively, and let 
$t>0$ and $\chi\in C(\t)$ be such that $\phi=tL+\chi$. 
Thanks to the above observation, $F_\phi(x,p,X)= tF(x,t^{-1}p,t^{-1}X)-\chi(x)$, the functions 
$v/t$ and $w/t$ are a subsolution and a supersolution of $\gl u+F[u]=t^{-1}(\chi+\eta)$ in $U$, respectively.  Obviously, we have $v/t\leq w/t$ on $\pl\gO$. 
Hence, 
by \erf{CP''}, we get $v/t\leq w/t$ in $U$. That is, we have 
$v\leq w$ in $U$. 
\end{proof} 

Now, for $z\in\t$ and $\gl\geq 0$, we define the sets $\cF_\gl\subset\Phi^+\tim C(\t)$, $\cG_\gl\subset C(\b)$, respectively, by 
\begin{align*}
&\cF_\gl:=\{(\phi,u)\in \Phi^+\tim C(\t)\mid u \ \text{is a subsolution of} \ \erf{DP'}\}, \\
&\cG_{z,\gl}:=\{\phi-\gl u(z)\mid (\phi,u)\in\cF_\gl\}
\end{align*}
as in Section \ref{sec2}. 
Since the set $\cG_{z,0}$ are independent of $z$, 
we write also $\cG_{0}$ for $\cG_{z,0}$. 
We note that we use here the same notation in Section \ref{sec2} by abuse of notation, 
although the symbols $\cF_0$ and $\cG_{z,\gl}$ denote the sets of functions 
different from those defined in Section \ref{sec2}.

\begin{lem}\label{thm3-sec3}Assume \erf{F1}, \erf{F2} 
and \erf{CP''}. For any 
$(z,\gl)\in\t\tim[0,\infty)$, the set $\cG_{z,\gl}$ is a convex cone with vertex at the origin.
\end{lem}
We can prove Lemma \ref{thm3-sec3} in a similar way to 
that of Lemma \ref{dmmc-thm1}, so we omit the proof.

Let $\cR_{L}$ denote the 
space of all Radon measures $\mu$ on $\b$ such that 
$\lan |\mu|, 1+|L|\ran<\infty$, where 
$|\mu|$ denotes the total variation of $\mu$, and, 
for any Borel subset $K$ of $\cA$, let $\cP_K$ denote 
the subset of all 
probability measures $\mu$ in $\cR_{L}$ that have support 
in $\t\tim K$. That is, $\cP_K=\{\mu\mid \mu 
\text{ is a Radon probability measure on }\b, \ \mu(\t\tim K)=1\}$.

For $(z,\gl)\in\t\tim[0,\,\infty)$, we define the set 
${\cG_{z,\gl}}'\subset\cR_{L}$ (denoted also by $\cG_0{'}$ if $\gl=0$) by 
\[
{\cG_{z,\gl}}':=\{\mu\in \cR_{L}\mid \lan\mu,f\ran\geq 0
\ \ \text{ for all }f\in\cG_{z,\gl}\},  
\]
and a key ingredient for our proof of Theorem \ref{thm2-sec3}
is stated as follows. 

\begin{thm}\label{thm4-sec3}Assume \erf{F1}, \erf{F2}, \erf{L},  
\erf{CP''}, \erf{EP}, and, if $\gl=0$,  \erf{Z} as well. Let 
$(z,\gl)\in\t\tim[0,\,\infty)$ and $v^\gl$ be 
a solution of \erf{DP}. Let $K_0$ be the compact set given by \erf{L}. 
If $K$ is a Borel subset of $\cA$ such that either 
$K\supsetneq K_0$ or $K=\cA$, then  
\begin{equation}\label{thm4-3-00}
\gl v^\gl(z)=\inf_{\mu\in\cP_{K}\,\cap\, {\cG_{z,\gl}}'}
\int_{\b}L(x,\ga)\mu(dx d\ga).
\end{equation}
\end{thm} 

In the theorem above, it is always possible
to choose a compact set $K\subset\cA$ so that either $K=\cA$ or $K\supsetneq K_0$. Indeed, if $\cA$ is compact, we may just choose $K=\cA$. Otherwise, we may choose 
$\ga_1\in\cA\setminus K_0$ and set  $K:=K_0\cup\{\ga_1\}$. Clearly, 
$K$ is compact and satisfies $K\supsetneq K_0$.  

A natural question is whether it is possible or not to 
replace $K$ by $K_0$ in formula \erf{thm4-3-00}. 
For instance, if $\cA$ is connected, then \erf{thm4-3-00} holds with $K_0$ in place of $K$. (See Remark \ref{rem2-26})

In the general situation where normalization \erf{Z} is dropped, 
we have the following formula 
for the critical value $c$.

\begin{cor}\label{cor2} Assume \erf{F1}, \erf{F2}, \erf{L},  
\erf{CP''}, and that \erf{EC} has a solution $(u,c)\in C(\t)\tim\R$. 
Let $K_0$ be the compact subset of $\cA$ from \erf{L}. 
Then, for any Borel set $K\subset \cA$ such that either 
$K\supsetneq K_0$ or $K=\cA$, 
we have 
\begin{equation}\label{thm4-3-000}
c=-\inf_{\mu\in\cP_{K}\,\cap\, {\cG_{0}}'}
\int_{\b}L(x,\ga)\mu(dx d\ga).
\end{equation}
\end{cor}

Conceding Theorem \ref{thm4-sec3}, we give a proof of Corollary 
\ref{cor2}. 

\begin{proof}  
Set $L_c:=L+c$ and note that $F_{L_c}=F-c$ and 
$u$ is a solution of \erf{EP'}, 
with $\phi=L_c$. Note that hypotheses \erf{F1}, \erf{F2}, \erf{L}
and \erf{CP''} hold, with $(F,\, L)$ replaced by $(F-c,L_c)$. 
Concerning \erf{L}, we note that the inequality 
\[
\inf_{\t\tim(\cA\setminus K_0)}L\geq L_0
=\max_{\t}\inf_{\ga\in K_0}L(x,\ga)
\]
implies 
\[
\inf_{\t\tim(\cA\setminus K_0)}L_c\geq L_0+c
=\max_{\t}\inf_{\ga\in K_0}L_c(x,\ga). 
\]

Fix any Borel subset $K$ of $\cA$ such that $K\supsetneq K_0$.
Observe that, if we set $\Phi_L^+:=\Phi$, then  
\[
\Phi_L^+=\{t L+\chi\mid t>0,\ \chi\in C(\t)\}
=\{tL_c+\chi\mid t>0,\, \chi\in C(\t)\}=\Phi_{L_c}^+,
\]
and, similarly, that $\,\cR_L=\cR_{L_c}$. 
The definitions of $\cF_0$, $\cG_0$, $\cG_0{'}$, and  $\cP_K$ 
depend on $L$ only through $\Phi^+$ and $\cR_L$, which 
assures that $\cF_0$, $\cG_0$, $\cG_0{'}$, and $\cP_K$  
are the same as those for $L_c$. Thus, we find 
from Theorem \ref{thm4-sec3} that
\[
0=\inf_{\mu\in\cP_K\cap\cG_0^{'}}\lan \mu, L_c\ran,
\]
which shows \erf{thm4-3-000}.  
\end{proof}

The next lemma is similar to Lemma \ref{thm7-sec2}.

\begin{lem}\label{conv-mu}Let $K$ be a compact subset of $\cA$ 
and $\{\mu_j\}_{j\in\N}\subset\cP_{K}$. 
Then there exists a subsequence $\{\mu_{j_k}\}_{k\in\N}$ of 
$\{\mu_j\}_{j\in\N}$ that converges to some $\mu\in\cP_{K}$ 
weakly in the sense of measure as $k\to\infty$ such that  
$\lan\mu_{j_k},L\ran \to \lan\mu, L\ran$ as $k\to\infty$.  
Furthermore, if $\mu_j\in\cG_{z,\gl_j}{'}$ for all $j\in\N$ 
and some 
$z\in\t$ and $\{\gl_j\}_{j\in\N}\subset[0,\infty)$ 
and if $\gl_j\to \gl$ as $j\to\infty$ for some $\gl\in[0,\,\infty)$, then $\mu\in \cG_{z,\gl}{'}$. 
\end{lem}

\begin{proof} Due to the obvious tightness of $\{\mu_j\}_{j\in\N}$, 
there exists its subsequence $\{\mu_{j_k}\}_{k\in\N}$ that converges,
weakly in the sense of measure, 
to a Radon probability measure $\mu$ on $\b$. It is clear that 
$\mu\in\cP_K$. 

Observe next that, for any $\eta\in C_{\rm c}(\cA)$ 
such that $\eta=1$ on $K$,
\[
\lan \mu_{j_k},L\ran=\lan \mu_{j_k},\eta L\ran \to 
\lan \mu, \eta L\ran=\lan \mu,L\ran \ \ \text{ as }\ k\to\infty.
\] 
 
Now, we fix $z\in\t$ and $\{\gl_j\}_{j\in\N}\subset[0,\infty)$, and assume that $\gl_j\to\gl$ as $j\to\infty$ for some $\gl\in\R$ and  
that $\mu_j\in\cG_{z,\gl_j}{'}$ for all $j\in\N$. Let 
$(tL+\chi,u)\in\cF_{\gl}$, where $t>0$ and $\chi\in C(\t)$, and  
observe that 
$(tL+\chi+(\gl_j-\gl)u,\,u)\in\cF_{\gl_j}$ for all $j\in\N$. 
 
Since $\mu_{j_k}\in\cG_{z,\gl_{j_k}}{'}$ and 
$(tL+\chi+(\gl_{j_k}-\gl)u,\,u)\in\cF_{\gl_{j_k}}$, 
we get 
\begin{equation}\label{ineq-k}
\lan\mu_{j_k},tL+\chi+(\gl_{j_k}-\gl)u-\gl_{j_k} u(z)\ran\geq 0  
\ \ \ \text{ for all }\ k\in\N.
\end{equation}
We have already seen that $\lan\mu_{j_k},L\ran \to \lan\mu,L\ran$ as 
$k\to\infty$. Also, 
since the functions $\chi+(\gl_{j_k}-\gl)u$ are 
bounded and continuous as a function on $\b$ and converge to $\chi$ 
uniformly on $\b$,  
we see that, as $k\to\infty$,  
\[
\lan \mu_{j_k},\chi+(\gl_{j_k}-\gl)u\ran \to \lan \mu,\chi\ran. 
\]
Hence, we find from \erf{ineq-k}, in the limit as $k\to\infty$, that $\lan \mu,t L+\chi-\gl u(z)\ran\geq 0$, which implies that 
$\mu\in\cG_{z,\gl}{'}$.  

The proof is now complete. 
\end{proof}

By Lemma \ref{conv-mu}, the set $\cP_{K}\,\cap\, {\cG_{z,\gl}}'$ 
is compact for any compact set $K\subset\cA$, 
and, under the hypotheses of Theorem \ref{thm4-sec3},  according to  Theorem \ref{thm4-sec3} and the consecutive remark, 
if the compact set $K\subset\cA$ is large enough in the sense of inclusion, then the set $\cP_{K}\,\cap\, {\cG_{z,\gl}}'$ is non-empty 
and the functional 
\[
\cP_{K}\,\cap\, {\cG_{z,\gl}}'\ni\mu\mapsto \int_{\b}L(x,\ga)\mu(dx d\ga)
\]
has a minimizer achieving the minimum value $\gl v^\gl(z)$.   
With an abuse of notation, we denote the set of such 
minimizers again by $\cM_{L,z,\gl}$ or, if $\gl=0$, by $\cM_{L,0}$.
We call any $\mu\in\cM_{L,0}$ (resp., $\mu\in\cM_{L,z,\gl}$ if $\gl>0$) 
a viscosity Mather measure (resp., a discounted Mather measure) associated with \erf{DP}. 
We point out that $\cM_{L,0}$ and $\cM_{L,z,\gl}$ may depend on our choice of $\Phi^+$ a priori. 
The authors do not know whether these are independent of the choice of 
$\Phi^+$ or not.

\begin{proof}[Proof of Theorem {\rm\ref{thm4-sec3}}] 
 Since $(L,v^\gl)\in\cF_{\gl}$,  for any $\mu\in\cP_{\cA}\cap\cG_{z,\gl}{'}$,
we have
\[
0\leq \lan \mu,L-\gl v^\gl(z)\ran=
\lan \mu,L\ran-\gl v^\gl(z),
\]
and, consequently, 
\begin{equation} \label{2-25-1}
\gl v^\gl(z)\leq 
\inf_{\mu\in\cP_{\cA}\,\cap\, {\cG_{z,\gl}}'}\lan \mu, L\ran.
\end{equation}

Next, to prove \erf{thm4-3-00}, we fix any Borel set 
$K\subset \cA$ such that either $K=\cA$ or $K\supsetneq K_0$, and 
show 
\begin{equation}\label{2-25-2}
\gl v^\gl(z)\geq 
\inf_{\mu\in\cP_{K}\,\cap\, {\cG_{z,\gl}}'}\lan \mu, L\ran,
\end{equation}
which, together with \erf{2-25-1}, yields \erf{thm4-3-00}. 

We set $K_1=\cA$ if $K_0=\cA$. Otherwise, 
we fix a point $\ga_1\in K\setminus K_0$
and set $K_1=K_0\cup \{\ga_1\}$. 
Notice that $K_1$ is a compact set and $K_1\subset K$.

To show \erf{2-25-2}, it is enough to prove 
\begin{equation}\label{2-26-1}
\gl v^\gl(z)\geq 
\inf_{\mu\in\cP_{K_1}\,\cap\, {\cG_{z,\gl}}'}\lan \mu, L\ran.
\end{equation}
We argue by contradiction, suppose that 
\[
\gl v^\gl(z)<
\inf_{\mu\in\cP_{K_1}\,\cap\, {\cG_{z,\gl}}'}\lan \mu, L\ran,
\]
select $\ep>0$ so that 
\begin{equation}\label{2-26-2}
\gl v^\gl(z)+\ep<
\inf_{\mu\in\cP_{K_1}\,\cap\, {\cG_{z,\gl}}'}\lan \mu, L\ran,
\end{equation}
and will get a contradiction.

We note by the cone property of $\cG_{z,\gl}$ (Lemma \ref{thm3-sec3}) that
\[
\inf_{f\in\cG_{z,\gl}}\lan \mu, f\ran=
\begin{cases}
0&\text{ if } \ \mu\in \cP_{K_1}\,\cap\, \cG_{z,\gl}{'}, \\[3pt]
-\infty  \ \ & \text{ if } \ \mu\in \cP_{K_1}\setminus {\cG_{z,\gl}}',
\end{cases}
\]
and, moreover, that
\[
\inf_{\mu\in\cP_{K_1}\,\cap\, \cG_{z,\gl}{'}}\lan \mu, L\ran
=\inf_{\mu\in\cP_{K_1}}\left(\lan \mu, L\ran
-\inf_{f\in\cG_{z,\gl}}\lan \mu, f\ran\right)
=\inf_{\mu\in\cP_{K_1}}\sup_{f\in\cG_{z,\gl}}
\lan \mu, L-f \ran. 
\] 
Noting by Lemma \ref{conv-mu} that  
$\cP_{K_1}$ is a convex compact space, we apply    
Sion's minimax theorem, to get 
\[
\inf_{\mu\in\cP_{K_1}}\sup_{f\in\cG_{z,\gl}}\lan\mu, L-f\ran
=\sup_{f\in\cG_{z,\gl}}\inf_{\mu\in\cP_{K_1}}\lan \mu,L-f\ran.
\]
Hence, we have 
\[
\inf_{\mu\in\cP_{K_1}\,\cap\, \cG_{z,\gl}{'}}\lan \mu, L\ran
=\sup_{f\in\cG_{z,\gl}}\inf_{\mu\in\cP_{K_1}}\lan \mu,L-f\ran.
\]

Combining the above with \erf{2-26-2} yields 
\begin{equation}\label{thm4-3-4}
\gl v^\gl(z)+\ep<
\sup_{f\in\cG_{z,\gl}}\inf_{\mu\in\cP_{K_1}}
\lan\mu, L-f\ran.
\end{equation} 
Hence, we may choose 
$(\phi,u)\in\cF_\gl$ such that 
\begin{equation} \label{thm4-3-5}\gl v^\gl(z)+\ep<
\lan \mu,L-\phi+\gl u(z)\ran \ \ \ \text{ for all }\mu\in\cP_{K_1},   
\end{equation}
and, moreover, $t>0$ and $\chi\in C(\t)$ so that $\phi=tL+\chi$.

We show that we can  
choose a constant $\gth>0$ such that $w:=\gth u$ is a 
subsolution of 
\begin{equation}\label{thm4-3-10}
\gl w+F[w]=\gl (w-v^\gl)(z)-\gth\ep\ \ \text{ in }\t.  
\end{equation}
Once this is done, we easily get a contradiction. 
Indeed, if $\gl>0$, then 
$\gz:=w+\gl (v^\gl-w)(z)+\gl^{-1}\gth\ep$ is a subsolution of $\gl \gz+F[\gz]=0$ in $\t$ 
and comparison principle \erf{CP''}  
yields $\gz\leq v^\gl$ in $\t$, which, after evaluation at $z$, 
gives $\,\gl^{-1}\gth\ep\leq 0$. This is a contradiction. 
On the other hand, if $\gl=0$, then, by Proposition \ref{exist1} (ii), we get $w+C\leq v^\gl$ in $\t$ for all $C\in\R$, which is a 
contradiction.

To show \erf{thm4-3-10}, 
we divide our argument into two cases. 
The first case is that when $K_1=\cA$.  
Then, we have 
$\gd_{(x,\ga)}\in\cP_{K_1}$ for all $(x,\ga)\in\b$, and, 
from \erf{thm4-3-4}, we get
\[
\gl v^\gl(z)+\ep<L-\phi+\gl u(z) 
\ \ \text{ on }\b,
\]
which yields 
\[
F(x,p,X)\leq F_\phi(x,p,X)+\gl (u-v^\gl)(z)-\ep \ \ \text{ for all }\ (x,p,X)
\in\t\tim\R^n\tim\bS^n. 
\]
Thus, we see that $w:=u$ is a subsolution of \erf{thm4-3-10}, with
$\gth=1$.

The second case is that when $K_1=K_0\cup\{\ga_1\}$ for some 
$\ga_1\in K\setminus K_0$.   
By assumption \erf{L}, we have  
\begin{equation}\label{m-1}
\max_{x\in\t}\min_{\ga\in K_0}L(x,\ga)\,=: L_0\,\leq \
\inf_{\t \tim (\cA\setminus K_0)}L. 
\end{equation}
Since $\gd_{(x,\ga)}\in\cP_{K_1}$ for all $(x,\ga)\in\t\tim K_1$, 
from \erf{thm4-3-5}, we get 
\begin{equation}\label{thm4-3-8}
(t-1)L(x,\ga)+\chi(x)<\gl(u- v^\gl)(z)-\ep \ \ \text{ for all }(x,\ga)\in\t\tim K_1. 
\end{equation}

We split our further argument into two cases. 
Consider first the case when $t\leq 1$, and, maximizing
the both sides of \erf{thm4-3-8} in $\ga\in K_0$ yields 
\[
(t-1)\min_{\ga\in K_0}L(x,\ga)+\chi(x)<\gl(u- v^\gl)(z)-\ep \ \ \text{ for all }x\in\t,
\] 
and, furthermore,
\[
(t-1)L_0+\chi(x)<\gl(u- v^\gl)(z)-\ep \ \ \text{ for all }x\in\t.
\]
Combining this with \erf{m-1}, we find that 
for all $(x,\ga)\in\t\tim (\cA\setminus K_0)$,
\[
(t-1)L(x,\ga)+\chi(x) 
\leq (t-1)L_0+\chi(x)<\gl(u- v^\gl)(z)-\ep,
\]
which shows, together with \erf{thm4-3-8}, that 
\[
(t-1)L(x,\ga)+\chi(x)<\gl(u- v^\gl)(z)-\ep 
\ \ \text{ for all }(x,\ga)\in\t\tim \cA. 
\] 
Hence, we get 
\[
\phi=L+(t-1)L+\chi<L+\gl(u- v^\gl)(z)-\ep \ \text{ on }\b,
\]
and find that $w:=u$ is 
a subsolution of \erf{thm4-3-10}, with $\gth=1$. 
Notice that the choice of $\ga_1$ is of no use so far.

Secondly, we consider the case when $t> 1$ and observe 
by the viscosity property of $v^\gl$ that, 
for any maximum point $x_0\in\t$ of $v^\gl$, 
which always exists, we have
\[\begin{aligned}
0&\,\geq \gl v^\gl(x_0)+ F(x_0,0,0)= \gl v^\gl(x_0)+ 
\sup_{\ga\in\cA}(-L(x_0,\ga))
\\&\,\geq \gl v^\gl(x_0)-\max_{x\in\t}
\inf_{\ga\in\cA}L(x,\ga)\geq \gl v^\gl (x_0)-L_0, 
\end{aligned}\]
and, hence, 
\begin{equation}\label{m-4}
\gl v^\gl(z)\leq \gl v^\gl(x_0)\leq L_0. 
\end{equation}

By \erf{m-1} and \erf{thm4-3-8}, 
we have 
\[L(x,\ga_1)\geq L_0 \ \ \text{ and } \ \  
(t-1)L(x,\ga_1)+\chi(x)<\gl(u- v^\gl)(z)-\ep \ \ \text{ for all }x\in\t, 
\]
which shows that 
\[
(t-1)L_0+\chi(x)<\gl(u- v^\gl)(z)-\ep \ \ \text{ for all }x\in\t.
\]
Hence, using \erf{m-4}, we get  
\[
(t-1)\gl v^\gl(z)+\chi\leq
(t-1)L_0+\chi<\gl(u- v^\gl)(z)-\ep,
\]
which yields 
\[
\chi<-t\gl v^\gl(z)+\gl u(z)-\ep \ \ \text{ on }\t,
\]
and moreover
\[
\phi=t L+\chi 
<tL+t\gl(u/t- v^\gl)(z)-\ep \ \ \text{ in }\b.
\]
Thus, we get formally  
\[
\gl u+\cL_\ga u\leq \phi<tL+t\gl(u/t- v^\gl)(z) -\ep \ \ \text{ in }\t\tim\cA.
\] 
From this, we deduce that $w:=u/t$ is a subsolution of 
\erf{thm4-3-10}, with $\gth=1/t$. This completes the proof. 
\end{proof} 

\begin{remark} \label{rem2-26}
In the proof above, the real requirement on the 
choice of $\ga_1$ is that the inequality $L(x,\ga_1)\geq L_0$ should hold 
for all $x\in\t$. 
A trivial sufficient condition for this is the condition that $\ga_1\in K\setminus K_0$. 
If there exists $\ga_1\in K_0$ such that $L(x,\ga_1)\geq L_0$ for all $x\in\t$, then we can perform the proof above, with $K_1=K_0\cup \{\ga_1\}=K_0$, which yields formula \erf{thm4-3-00} with $K_0$ in place of $K$. 
For instance, if $\cA$ is connected, then \erf{L} guarantees the 
existence of such $\ga_1\in K_0$. To see this, one defines 
the continuous function $l$ on $\cA$ by setting $l(\ga)=\min_{x\in\t}L(x,\ga)$ and observes that $l(\cA)$ is an interval of $\R$, 
$l(\cA\setminus K_0)\subset [L_0,\,\infty)$ and $l(K_0)$ is compact, which is impossible if 
$l(K_0)\subset (-\infty,\, L_0)$.   
\end{remark}

The following proof of Theorem \ref{thm2-sec3} is parallel to 
that of Theorem \ref{thm1-sec2}, but we still give it here for the reader's convenience.

\begin{proof}[Proof of Theorem {\rm\ref{thm2-sec3}}] Let $(u_0,c)\in C(\t)\tim\R$ 
be a solution of \erf{EP}. The existence of such a solution and the uniqueness of the critical value $c$ is assured by Proposition 
\ref{prop-fund}. 
 
As we have seen in the proof of Theorem \ref{thm1-sec2}, the proof in 
the general case regarding $c$ is easily reduced to the special 
case when $c=0$. Henceforth, we are thus only concerned with the case when $c=0$.

As in the proof of Theorem \ref{thm1-sec2}, we see that 
$|v^\gl|\leq 2\|u_0\|_{C(\t)}$ on $\t$ for all $\gl>0$. This together with 
\erf{EC} guarantees that $\{v^\gl\}_{\gl>0}$ is relatively compact 
in $C(\t)$. Therefore, it is enough to show that
\begin{equation}\label{thm2-sec3-1}
\limsup_{\gl\to 0+}v^\gl(x)=\liminf_{\gl\to 0+}v^{\gl}(x) 
\ \ \text{ for all }\ x\in\t.
\end{equation}

To check the above convergence, we fix any $z\in\t$ and select 
two sequences $\{\gl_j\}_{j\in\N}\subset(0,\,1)$ and $\{\gd_j\}_{j\in\N}\subset(0,\,1)$, converging to zero, 
so that 
\[
\liminf_{\gl\to 0+}v^\gl(z)=\lim_{j\to\infty}v^{\gl_j}(z)
\ \ \ \text{ and } \ \ \ 
\limsup_{\gl\to 0+}v^\gl(z)=\lim_{j\to\infty}v^{\gd_j}(z).
\]
By passing to subsequences if needed, we may assume that 
for some $v,\,w\in C(\t)$,  
\[
\lim_{j\to\infty}v^{\gl_j}=v \ \ \text{ and } \ \ 
\lim_{j\to\infty}v^{\gd_j}=w 
\ \ \text{ in }\  C(\t). 
\] 

Fix $z\in \t$, and, for each $j\in\N$, thanks to Theorem \ref{thm4-sec3}, we choose  
a discounted Mather measure $\mu_j\in\cM_{L,z,\gl_j}$, so that 
$\lan\mu_j,L\ran=\gl_j v^{\gl_j}(z)$. We may assume that 
$\{\mu_j\}_{j\in\N}\subset \cP_{K}$ for some compact set 
$K\subset\cA$. 
We deduce by Lemma \ref{conv-mu} that we may assume, after passing to a subsequence if necessary, that $\{\mu_j\}$ converges to 
some $\mu\in\cM_{L,0}$.   
 
By the stability property of the viscosity solutions under uniform convergence, 
$w$ is a solution of $F[u]=0$ in $\t$. Hence, we find that 
$(L-\gd_j v^{\gd_j},v^{\gd_j})\in\cF_0$ and $(L+\gl_j w,w)\in\cF_{\gl_j}$, and, consequently, we get 
\[
0\leq \lan\mu, L-\gd_j v^{\gd_j}\ran =-\gd_j \lan \mu,v^{\gd_j}\ran, 
\]
and  
\[
0\leq \lan\mu_j, L+\gl_j w-\gl_j w(z)\ran
=\gl_j (v^{\gl_j}-w)(z)+\gl_j\lan \mu_j,w\ran.
\]
From these, we get, in the limit as $j\to\infty$,
\[
\lan \mu,w\ran\leq 0 \ \ \text{ and } \ \ 0\leq (v-w)(z)+\lan \mu,w\ran,
\]
which yields $w(z)\leq v(z)$, Since $z\in\t$ is arbitrary, we get 
$w\leq v$ in $\t$. This shows that \erf{thm2-sec3-1} holds. 
\end{proof}

\section{Examples} \label{example}

Theorems \ref{thm1-sec2} and \ref{thm2-sec3} can be applied under the 
structure conditions given in \cite{CrIsLi, CaLePo, ArTr}. 
Let us give some representative examples here.

\subsection{First-order Hamilton-Jacobi equations with linear growth in the gradient variable}
Consider the case that
\[
F(x,p,X):=\max_{|\alpha|\le 1}\{-c(x)\alpha\cdot p-L(x,\alpha)\}, 
\]
where $\cA=\overline{B}(0,1)$, $c\in C(\t)$ with $c>0$, and $L(x,\alpha)\in C(\b)$. 
Then \erf{DP}, \erf{EP} are first-order Hamilton-Jacobi equations with linear growth which arises 
in the context of front propagations, crystal growth, and optimal control problems. 
We can easily check that \erf{F1}--(F3), \erf{CP'}, \erf{EC} hold, 
and so the conclusion of Theorem \ref{thm1-sec2} is valid in this case.

\subsection{Viscous Hamilton-Jacobi equations with superlinear growth in the gradient variable}
Consider the function $F$ of the form
\[
F(x,p,X):= - \tr\gs^t(x)\gs(x) X +H(x,p)
\]
where $\sigma \in \Lip(\t, \R^{k \times n})$ for some $k \in \N$ and the Hamiltonian $H: \t \times \R^n \to \R$ satisfies
\begin{itemize}
\item  $p \mapsto H(x,p)$ is convex for all $x\in \t$,

\item there exist $m>1$ and $\Lambda\geq 1$ such that, for every $x,y\in \t$ and $p,q \in \R^n$,
\begin{align}
&\frac{1}{\Lambda}|p|^m - \Lambda \leq H(x,p) \leq \Lambda (|p|^m +1), \label{vis-co1}\\
&|H(x,p)-H(y,p)| \leq \Lambda (|p|^m+1) |x-y|, \label{vis-co2}\\
&|H(x,p)-H(x,q)| \leq \Lambda (|p|+|q|+1)^{m-1} |p-q|. \label{vis-co3}
\end{align}
\end{itemize}
Let $L(x,q)$ be the Legendre transform of $H(x,p)$, i.e., $L(x,q):=\sup_{p\in\R^n}\{p\cdot q-H(x,p)\}$. 
We can rewrite $F$ as 
\[
F(x,p,X) = \sup_{q \in \R^n} \left( - \tr\gs^t(x)\gs(x) X - q\cdot p - L(x,q) \right).
\]
Note that this general situation includes first-order Hamilton-Jacobi equations with 
superlinear growth (the case where $\sigma \equiv 0$).

It is clear that $F$ satisfies \erf{L} because of assumption \eqref{vis-co1} with $m>1$.
The comparison principle \erf{CP''} holds in light of \cite[Theorem 2.1]{ArTr}.
Assumption \erf{EC} holds thanks to \cite[Theorem 3.1]{ArTr}.
In fact, we have further that $\{v^\gl\}_{\gl>0}$ is equi-Lipschitz continuous with Lipschitz constant 
dependent only on $\Lambda, \Lip(\sigma), \|H(\cdot,0)\|_{C(\T^n)}$ (see also \cite{CaLePo}).
It is worth mentioning that \erf{vis-co2} is a bit more general than the classical condition (3.14) in \cite{CrIsLi}. 
For example, if $H(x,p)=c(x)|p|^2$ with $c \in \Lip(\t)$ and $c>0$, then \erf{vis-co1}--\erf{vis-co3} hold, but not  (3.14) in \cite{CrIsLi}. 

\subsection{Fully nonlinear, uniformly elliptic equations with linear growth in the gradient variable}
Assume $\cA$ is compact and 
\[
F(x,p,X):= \max_{\ga \in \cA} \left(- \tr a(x,\ga) X -b(x,\ga)\cdot p-L(x,\ga)\right), 
\]
where $a\in \Lip(\t \times \cA, \bS^n)$, $b\in \Lip(\t \times \cA,\R^n)$, $L\in C(\t \times \cA,\R)$.
We assume further that
there exists $\theta>0$ such that 
\[
\frac{1}{\theta} I \le a(x,\ga)\le \theta I \quad\text{for all} \ (x,\ga) \in\t \times \cA. 
\]
Here, $I$ denotes the identity matrix of size $n$. 

Clearly \erf{F1}--(F3) hold. Comparison, \erf{CP'}, is a 
consequence of \cite[Theorem III.1 (1)]{IsLi}. The Lipschitz 
continuity of $a$ and $b$ is assumed as a sufficient 
condition that \erf{CP'} holds. This Lipschitz continuity assumption on $a$ and $b$ can be replaced by a weaker assumption as in \cite[Theorem III.1 (1)]{IsLi}.

In view of the Krylov-Safonov estimate (see \cite{KrSa}and also \cite{Ca,Tr}), we have that
$\{v^\lambda\}_{\lambda>0}$ is equi-H\"older continuous, which implies (EC). Here are some details on this H\"older estimate. 
According to Proposition \ref{exist1} (i), there exists a constant $M_0>0$ such that $\gl |v^\gl|\leq M_0$ on $\t$, which implies that 
$v^\gl$ satisfies, in the viscosity sense,
\[
F[v^\gl]\leq M_0   \ \ \ \text{ and } \ \ \ F[v^\gl]\geq -M_0 
\ \ \ \text{ in } \ \t.
\]
According to \cite[(2.14)]{Tr}, there exist constants 
$\gamma\in (0,\,1)$ and $C>0$, depending only on 
$n$, $\gth$ and $\|b\|_{C(\b)}$, such that 
\begin{equation} \label{ue1}
\osc_{B_{\tau R}(x)}v^\gl\leq C\tau^\gamma 
\left(\osc_{B_R(x)} v^\gl+\gth M_0 R^2\right) 
\ \ \text{ for all }\ \tau\in(0,\,1),\ x\in\R^n \ \text{ and }\ R>0,
\end{equation}
where $v^\gl$ is regarded as a periodic function on $\R^n$ and 
$\osc_{B_r(x)} v^\gl:=\sup_{B_r(x)}v^\gl-\inf_{B_r(x)}v^\gl$.  
We select $\tau_0\in (0,\,1)$ so small that $C\tau_0^\gamma\leq \fr 12$, 
then select $R_0>0$ so large that $[-\fr12,\,\fr 12]^n\subset B_{\tau_0 R_0}$,
and note that 
\[
\osc_{\T^n}v^\gl=\osc_{B_{\tau_0 R_0}}v^\gl=\osc_{B_{R_0}}v^\gl, 
\] 
to find from \erf{ue1} that 
\[
\osc_{\t} v^\gl\leq \fr 12\left(\osc_{\t}v^\gl+\gth M_0R_0^2\right)
\]
and hence
\[
\osc_{\t}v^\gl \leq \gth M_0 R_0^2.
\]
Thus, \erf{ue1} yields 
\[
\osc_{B_{\tau R_0}(x)}v^\gl\leq C_1\tau^\gamma 
\ \ \text{ for all }\ \tau\in(0,\,1)\ \text{ and } \ x\in\R^n,
\] 
where $C_1:=2C\gth M_0 R_0^2$, 
which shows that $\{v^\gl\}_{\gl>0}$ is 
equi-H\"older continuous on $\t$. 

Owing to the Evans-Krylov theorem (see \cite{Ev, Kr}), 
if $F$ is locally H\"older continuous on $\t\tim\R^n\tim\bS^n$, then \erf{DP} has a unique classical solution. This allows us to formulate 
``Mather measures'' based on classical solutions (see also \cite{Go}). For instance, 
we may replace $\cF_0$ of Section \ref{sec2} by 
\[
\{(\phi,u)\in C(\b)\tim C^2(\t)\mid 
u \ \text{ is a subsolution of }\ \erf{EP'}\}. 
\] 
We leave 
a further discussion of this approach to the interested readers.

\subsection{Fully nonlinear equations with superquadratic growth in the gradient variable}
Assume that $F$ has the form
\[
F(x,p,X) := \sup_{q \in \R^n} \left( -\tr \gs^t(x,q)\gs(x,q) X - q\cdot p -L(x,q) \right)
\]
and 
\begin{itemize}

\item there exist $m>2$ and $\Lambda>0$ such that, for every $x \in \t$, $p \in \R^n$ and $X \in \bS^n$,
\[
F(x,p,X) \geq \frac{1}{\Lambda} |p|^m - \Lambda(|X|+1),
\]

\item the structural condition \cite[(3.14)]{CrIsLi} holds, that is, 
there exists a modulus of continuity $\omega:[0,\infty) \to [0,\infty]$ satisfying $\omega(0+)=0$  such that
\begin{equation}\notag
\left\{\text{
\begin{minipage}{0.85\textwidth}
$F(y,\alpha(x-y),Y) - F(x,\alpha(x-y), X) \leq \omega(\alpha|x-y|^2+|x-y|)$\\
whenever $x,y \in \t$, $\alpha>0$, $X,Y \in \bS^n$, and\\
$-3 \alpha \left(\begin{array}{cc}I&0\\0&I\end{array}\right) \leq  \left(\begin{array}{cc}X&0\\0&-Y\end{array}\right) \leq 3 \alpha \left(\begin{array}{cc}I&-I\\-I&I\end{array}\right).$
\end{minipage}
}\right.\end{equation}

\end{itemize}
Since $m>2$, \erf{L} holds true.
 The comparison principle \erf{CP''} follows from the comparison principle in \cite{CrIsLi}.
The family $\{v^\lambda\}_{\lambda>0}$ is equi-H\"older continuous with H\"older exponent $\gamma =\frac{m-2}{m-1}$ in view of 
\cite[Theorem 1.1]{CaLePo}, \cite[Lemma 3.2]{ArTr}.

We refer the reader to \cite[Example 3.6]{CrIsLi} for some further examples of $F$ satisfying (3.14) there.
Let us also point out that, because of the a priori $\gamma$-H\"older estimates for subsolutions, 
we can relax \cite[(3.14)]{CrIsLi} a bit more. 
More precise, we can replace it by \cite[(5.1)]{CrIsLi}, that is, 
there exist a modulus of continuity $\omega:[0,\infty) \to [0,\infty]$ satisfying $\omega(0+)=0$  and some $\theta>2-\gamma$ such that
\begin{equation}\notag
\left\{\text{
\begin{minipage}{0.85\textwidth}
$F(y,\alpha(x-y),Y) - F(x,\alpha(x-y), X) \leq \omega(\alpha|x-y|^\theta+|x-y|)$\\
whenever $x,y \in \t$, $\alpha>0$, $X,Y \in \bS^n$, and\\
$-3 \alpha \left(\begin{array}{cc}I&0\\0&I\end{array}\right) \leq  \left(\begin{array}{cc}X&0\\0&-Y\end{array}\right) \leq 3 \alpha \left(\begin{array}{cc}I&-I\\-I&I\end{array}\right).$
\end{minipage}
}\right.\end{equation}

\section*{Appendix} \label{app}
We here provide proofs of Propositions \ref{exist1} and \ref{prop-fund} in Introduction.

\begin{proof}[Proof of Proposition {\rm\ref{exist1}}]  (i)
 We set $M_0=\max_{x\in\t}|F(x,0,0)|$, observe that 
the constant functions $M_0/\gl$ and $-M_0/\gl$ are a supersolution 
and a subsolution of \erf{DP}, and apply the Perron method, to find 
a solution $v^\gl$ of \erf{DP} which satisfies $|v^\gl|\leq M_0/\gl$ 
on $\t$. Here, by solution, we mean that 
the upper semicontinuous envelope $(v^\gl)^*$ 
and the lower semicontinuous envelope $(v^\gl)_*$ of $v^\gl$ 
are a subsolution and a supersolution of \erf{DP}, respectively. Then, by the comparison principle, \erf{CP}, we get $(v^\gl)^*\leq (v^\gl)_*$ 
on $\t$, which implies that $v^\gl\in C(\t)$. The uniqueness of 
solutions of \erf{DP} in $C(\t)$ is also 
an immediate consequence of \erf{CP}. Clearly we have $\gl|v^\gl|\leq M_0$ for all $\gl>0$.  

(ii)  Assume that 
$c_1<c_2$ and $v\leq w$ on $\pl U$,  
set $c=(c_1+c_2)/2$, and choose $\gd>0$ small enough so that
$v$ and $w$ are a subsolution and a supersolution of $\gd u+F[u]= c$
in $U$, respectively. Then, we observe that the functions 
$v-\gd^{-1}c$ and $w-\gd^{-1}c$ are a subsolution and a supersolution of $\gd u+F[u]=0$
in $U$, respectively, and 
we use \erf{CP} 
to deduce that $v\leq w$ in $U$.
\end{proof}

\begin{proof}[Proof of Proposition {\rm\ref{prop-fund}}]
 The existence and uniqueness of $v^\gl\in C(\t)$ has 
already been
established in Proposition \ref{exist1} (i). 
Moreover, we have $\gl|v^\gl|\leq M_0$ on $\t$ for all $\gl>0$ 
and for $M_0=\max_{x\in\t}|F(x,0,0)|$. 

Now, we set $m^\gl:=\min_{\t}v^\gl$ and $u^\gl:=v^\gl-m^\gl$ 
and note by assumption \erf{EC} that the family $\{u^\gl\}_{\gl>0}$ is equi-continuous and uniformly bounded on $\t$. The Arzela-Ascoli theorem 
assures that there is a sequence $\{\gl_j\}_{j\in\N}$ of positive numbers 
converging to $0$ such that $u^{\gl_j}\to u$ in $C(\t)$ as $j\to\infty$
for some $u\in C(\t)$. Since $\gl|v^\gl|\leq M_0$ for all 
$\gl>0$, the sequence $\{\gl_j m^{\gl_j}\}_{j\in\N}\subset\R$ 
is bounded and, hence, has a convergent subsequence. Replacing 
$\{\gl_j\}_{j\in\N}$ by such a subsequence, there exists a constant $c\in\R$ such that
\[
\lim_{j\to \infty}\gl_j m^{\gl_j}=-c. 
\]  

Noting that $u^\gl$ is a solution of $\gl u^\gl + F[u^\gl]=-\gl m^\gl$
in $\t$, we deduce by the stability of viscosity properties 
under uniform convergence that $u$ is a solution of $F[u]=c$ in $\t$. 
That is, $(u,c)$ is a solution of \erf{EP}. 

To see the uniqueness of critical value for \erf{EP}, let 
$(v,d),\, (w,e)\in C(\t)\tim\R$ be two solutions of \erf{EP}, 
suppose that $d\not=e$ and get a contradiction. We may assume without 
loss of generality that $d<e$. Noting that, for any constant $C\in\R$, $(v+C,d)$ is a solution of \erf{EP} and applying Proposition \ref{exist1} (ii), we see that $v+C\leq w$ on $\t$ for all $C\in\R$. 
This is a contradiction, which shows that the critical value $c$ for 
\erf{EP} is unique. 

What remains is to show \erf{i-1}. By the equi-continuity, \erf{EP}, 
we infer that
\[
\sup_{\gl>0}\left(\max_{\t}v^\gl-\min_{\t} v^\gl\right)<\infty.
\] 
Consequently, 
\[
\lim_{j\to\infty}\gl_j\max_{\t}v^{\gl_j}=\lim_{j\to\infty}\gl_j
m^{\gl_j}=-c.
\]
Thus we find that 
\[
c=-\lim_{j\to\infty}\gl_j v^{\gl_j}(x) \ \ \text{ in }\ C(\t).
\]

Arguing as above, we see that every 
sequence $\{\gd_j\}_{j\in\N}$ of positive numbers convergent to $0$
has a subsequence, which we denote again by the same symbol,  
a constant $d\in\R$ and a function $w\in C(\t)$ such that 
\[
d=-\lim_{j\to\infty} \gd_j v^{\gd_j} \ \ \text{ in }\ C(\t),
\]   
\[
w=\lim_{j\to \infty} \left(v^{\gd_j}-\min_{\t}v^{\gd_j}\right)
\ \ \text{ in }\ C(\t),
\]
and $(w,d)$ is a solution of \erf{EP}.  Here, 
because of the uniqueness of the critical value $c$, 
we find that $d=c$. We now easily deduce that \erf{i-1} holds, and the proof is complete.
\end{proof}



\begin{bibdiv}
\begin{biblist}

\bib{AlAlIsYo}{article}{
   author={Al-Aidarous, E. S. },
   author={Alzahrani, E. O. },
   author={Ishii, H.},
   author={Younas, A. M. M. },
   title={A convergence result for the ergodic problem for
Hamilton-Jacobi equations with Neumann-type
boundary conditions},
   journal={Proc. Royal Soc. Edinburgh},
   volume={146A},
   date={2016},
   pages={225--242},
   issn={},
   review={},
}
\bib{Ar}{article}{
   author={Armstrong, Scott N.},
   title={The Dirichlet problem for the Bellman equation at resonance},
   journal={J. Differential Equations},
   volume={247},
   date={2009},
   number={3},
   pages={931--955},
   issn={0022-0396},
}


\bib{ArTr}{article}{
   author={Armstrong, S. N.},
   author={Tran, H. V.},
   title={Viscosity solutions of general viscous Hamilton-Jacobi equations},
   journal={Math. Ann.},
   volume={361},
   date={2015},
   number={3-4},
   pages={647--687},
   issn={0025-5831},
}


\bib{BCD1997}{book}{
   author={Bardi, M.},
   author={Capuzzo-Dolcetta, I.},
   title={Optimal control and viscosity solutions of Hamilton-Jacobi-Bellman
   equations},
   series={Systems \& Control: Foundations \& Applications},
note={With appendices by Maurizio Falcone and Pierpaolo Soravia},
   publisher={Birkh\"auser Boston, Inc., Boston, MA},
   date={1997},
   pages={xviii+570},
   isbn={0-8176-3640-4},
}
\bib{Ca}{article}{
   author={Caffarelli, Luis A.},
   title={Interior a priori estimates for solutions of fully nonlinear
   equations},
   journal={Ann. of Math. (2)},
   volume={130},
   date={1989},
   number={1},
   pages={189--213},
   issn={0003-486X},
}
\bib{CaLePo}{article}{
   author={Capuzzo-Dolcetta, I.},
   author={Leoni, F.},
   author={Porretta, A.},
   title={H\"older estimates for degenerate elliptic equations with coercive
   Hamiltonians},
   journal={Trans. Amer. Math. Soc.},
   volume={362},
   date={2010},
   number={9},
   pages={4511--4536},
}

\bib{CrIsLi}{article}{
   author={Crandall, M. G.},
   author={Ishii, H.},
   author={Lions, P.-L.},
   title={User's guide to viscosity solutions of second order partial
   differential equations},
   journal={Bull. Amer. Math. Soc. (N.S.)},
   volume={27},
   date={1992},
   number={1},
   pages={1--67},
}
\bib{DFIZ}{article}{
   author={Davini, A.},
   author={Fathi, A.},
   author={Iturriaga,  R.},
   author={Zavidovique, M.},
   title={Convergence of the solutions of the discounted Hamilton-Jacobi equation},
   journal={Invent. Math.},
   volume={{\rm First online}},
   date={2016},
   number={},
   pages={1--27},
}
\bib{DV1}{article}{
   author={Donsker, Monroe D.},
   author={Varadhan, S. R. S.},
   title={On a variational formula for the principal eigenvalue for
   operators with maximum principle},
   journal={Proc. Nat. Acad. Sci. U.S.A.},
   volume={72},
   date={1975},
   pages={780--783}, 
}

\bib{DV2}{article}{
   author={Donsker, M. D.},
   author={Varadhan, S. R. S.},
   title={On the principal eigenvalue of second-order elliptic differential
   operators},
   journal={Comm. Pure Appl. Math.},
   volume={29},
   date={1976},
   number={6},
   pages={595--621},
}

\bib{Ev}{article}{
   author={Evans, Lawrence C.},
   title={Classical solutions of fully nonlinear, convex, second-order
   elliptic equations},
   journal={Comm. Pure Appl. Math.},
   volume={35},
   date={1982},
   number={3},
   pages={333--363},
   issn={0010-3640},
}
\bib{Go}{article}{
   author={Gomes, D. A.},
   title={Duality principles for fully nonlinear elliptic equations},
   conference={
      title={Trends in partial differential equations of mathematical
      physics},
   },
   book={
      series={Progr. Nonlinear Differential Equations Appl.},
      volume={61},
      publisher={Birkh\"auser, Basel},
   },
   date={2005},
   pages={125--136},
}
\bib{IsLi}{article}{
   author={Ishii, H.},
   author={Lions, P.-L.},
   title={Viscosity solutions of fully nonlinear second-order elliptic
   partial differential equations},
   journal={J. Differential Equations},
   volume={83},
   date={1990},
   number={1},
   pages={26--78},
}
\bib{IsMiTr2}{article}{ 
   author={Ishii, H.},
   author={Mitake, H.},
   author={Tran, H. V.},
   title={work in progress},
}
\bib{Kr}{article}{
   author={Krylov, N. V.},
   title={Boundedly inhomogeneous elliptic and parabolic equations},
   language={Russian},
   journal={Izv. Akad. Nauk SSSR Ser. Mat.},
   volume={46},
   date={1982},
   number={3},
   pages={487--523, 670},
}
\bib{KrSa}{article}{
   author={Krylov, N. V.},
   author={Safonov, M. V.},
   title={An estimate for the probability of a diffusion process hitting a
   set of positive measure},
   language={Russian},
   journal={Dokl. Akad. Nauk SSSR},
   volume={245},
   date={1979},
   number={1},
   pages={18--20},
}

\bib{LPV}{article}{ 
   author={P.-L. Lions},
   author={G. Papanicolaou},
   author={S. R. S. Varadhan},
   title={Homogenization of Hamilton--Jacobi equations},
   journal={unpublished work (1987)},
}

\bib{Man}{article}{
   author={R. Ma\~n\'e},
   title={Generic properties and problems of minimizing measures of Lagrangian systems},
   journal={Nonlinearity},
   volume={9},
   date={1996},
   NUMBER = {2},
   pages={273--310},
}

\bib{Mat}{article}{
   author={J. N. Mather},
   title={Action minimizing invariant measures for positive definite Lagrangian systems},
   journal={Math. Z.},
   volume={207},
   date={1991}, 
   NUMBER = {2},
   pages={169--207},
}

\bib{MiTr}{article}{ 
   author={Mitake, H.},
   author={Tran, H. V.},
   title={Selection problems for a discounted degenerate viscous   Hamilton--Jacobi equation},
   journal={submitted, (Preprint is available in arXiv:1408.2909)},
}



\bib{Si}{article}{
   author={Sion, M.},
   title={On general minimax theorems},
   journal={Pacific J. Math.},
   volume={8},
   date={1958},
   pages={171--176},
}
\bib{Te}{article}{
   author={Terkelsen, F.},
   title={Some minimax theorems},
   journal={Math. Scand.},
   volume={31},
   date={1972},
   pages={405--413 (1973)},
}
\bib{Tr}{article}{
   author={Trudinger, Neil S.},
   title={On regularity and existence of viscosity solutions of nonlinear
   second order, elliptic equations},
   conference={
      title={Partial differential equations and the calculus of variations,
      Vol.\ II},
   },
   book={
      series={Progr. Nonlinear Differential Equations Appl.},
      volume={2},
      publisher={Birkh\"auser Boston, Boston, MA},
   },
   date={1989},
   pages={939--957},
}

\end{biblist}
\end{bibdiv}
\bye